\documentclass[reqno,11pt]{amsart}
\usepackage{amsthm,amsfonts,amssymb,euscript,  
mathrsfs,graphics,color,amsmath,latexsym,marginnote}  
\usepackage{cite}       
\usepackage{dsfont}           
\usepackage[english]{babel}  
\usepackage{mathtools}  
\usepackage{times}

\numberwithin{equation}{section}

\usepackage{datetime}
\usepackage{lipsum}
\usepackage{enumerate}
\usepackage{tikz}
\usetikzlibrary{matrix}
\usepackage{hyperref}
\RequirePackage{amscd}
\RequirePackage{epic}
\RequirePackage[all]{xy}
\RequirePackage{url}
\RequirePackage[shortlabels]{enumitem}
\usepackage{empheq}
\usepackage{epsfig}
\usepackage{lineno} 
 \usepackage{perpage}
 \usepackage[english]{babel}

\usepackage[utf8]{inputenc}
\usepackage{cite}
\RequirePackage{amscd}
\RequirePackage{epic}
\RequirePackage{eepic}
\RequirePackage[all]{xy}
\RequirePackage{url}
\RequirePackage[shortlabels]{enumitem}
\usepackage{empheq}
\usepackage{nomencl}
\usepackage{epsfig}
\usepackage{graphicx}

\theoremstyle{plain}

\newtheorem{theorem}{Theorem}[section]
\newtheorem{proposition}[theorem]{Proposition}

\newtheorem{lemma}[theorem]{Lemma}

\newtheorem{remark}[theorem]{Remark}

\newtheorem{definition}[theorem]{Definition}

\newcommand{\na}{\widehat{n}}

\newcommand{\pa}[1]{{\left(#1\right)}}

    \newcommand{\set}[1]{{\left\{#1\right\}}}
\newcommand{\norm}[1]{{\left |#1\right |}}

\RequirePackage{epic}
\RequirePackage{eepic}
\RequirePackage[all]{xy}
\RequirePackage{url}
\RequirePackage[shortlabels]{enumitem}
\usepackage{empheq}
\usepackage{nomencl}
\usepackage{epsfig}
\usepackage{graphicx}

\newcommand{\T}{\mathbb{T}}
\newcommand{\Z}{\mathbb{Z}}
\newcommand{\R}{\mathbb{R}}
\newcommand{\C}{\mathbb{C}}

\newcommand{\co}[1]{\textit{#1}}
\newcommand{\gr}[1]{\textbf{#1}}

\usepackage{amsthm}

\newcommand{\s}{{\sigma}}
\newcommand{\ii}{{\rm i}}

\def\wc{ {}}

\newcommand{\N}{{\mathbb N}}

\newcommand{\cF}{{\mathcal F}}

\newcommand{\cH}{{\mathcal H}}

\newcommand{\cK}{{\mathcal K}}

\newcommand{\cM}{{\mathcal M}}

\newcommand{\cP}{{\mathcal P}}

\newcommand{\cR}{{\mathcal R}}

\newcommand{\td}{{\mathtt{d}}}

\newcommand{\tm}{{\mathtt{m}}}

\newcommand{\tr}{{\mathtt{r}}}

\newcommand{\tD}{{\mathtt{D}}}

\newcommand{\tN}{{\mathtt{N}}}

\usepackage{bm}

\newcommand{\al}{{\alpha}}
\newcommand{\bt}{{\beta}}

\newcommand{\abs}[1]{\left| #1 \right|}

\newcommand{\jap}[1]{\langle #1 \rangle}
\newcommand{\und}[1]{\underline{#1}}

\newcommand{\e}{{\varepsilon}}

\newcommand{\meas}{{\mathtt{meas}}}

\newcommand{\tw}{{\mathtt{w}}}
\renewcommand{\th}{{\mathtt{h}}}
\newcommand{\nnorm}[1]{{\left\vert\kern-0.25ex\left\vert\kern-0.25ex\left\vert #1 
    \right\vert\kern-0.25ex\right\vert\kern-0.25ex\right\vert}}

\usepackage{framed,enumitem} 

\newcommand{\ri}{r}

\newcommand{\rf}{{r'}}

\newcommand{\twi}{\tw}
\newcommand{\twf}{{\tw'}}

\newcommand{\imm}{{\rm{i}}}

\definecolor{aqua}{RGB}{10,150,200}

\makeatletter

\def\l@subsection{\@tocline{2}{0pt}{2.5pc}{5pc}{}}
\def\l@subsubsection{\@tocline{3}{0pt}{4.5pc}{5pc}{}}

\begin{document}

\title{Non-resonant conditions for the Klein-Gordon equation on the circle
}

\date{}

\author{Roberto Feola}
\address{\scriptsize{Dipartimento di Matematica e Fisica, Universit\`a degli Studi RomaTre, 
Largo San Leonardo Murialdo 1, 00144 Roma}}
\email{roberto.feola@uniroma3.it}

\author{Jessica Elisa Massetti}
\address{\scriptsize{Dipartimento di Matematica, Universit\`a degli Studi di Roma ``Tor Vergata", 
Via della Ricerca Scientifica 1, 00133 Roma}}
\email{massetti@mat.uniroma2.it}

 


\begin{abstract}    
We consider the infinite dimensional vector of frequencies $\omega(\mathtt{m})=( \sqrt{j^2+\mathtt{m}})_{j\in \Z}$, $\mathtt{m}\in [1,2]$
arising form a linear Klein-Gordon equation on the one dimensional torus and prove that there exists a positive measure set of masses $\mathtt{m}'$s for which $\omega(\mathtt{m})$ satisfies
a diophantine condition similar to the one introduced by Bourgain in \cite{Bourgain:2005}, 
in the context of Schr\"odinger equation with convolution potential.
The main difficulties we have to deal with are  
the asymptotically linear nature of the (infinitely many) $\omega_{j}'$s and the degeneracy coming from having 
only one parameter at disposal for their modulation.
As an application we provide estimates on the inverse of the adjoint action of the associated quadratic Hamiltonian 
on homogenenous polynomials of any degree in Gevrey category.
\end{abstract}

\maketitle 
\setcounter{tocdepth}{2}
\tableofcontents

\section{Introduction} 
The study of existence of periodic/quasi-periodic motions plays a pivotal role in the general understanding
of the evolution of a dynamical system.
The relevance of such special solutions, in the context of $n$-dimensional Hamiltonian systems, $n\geq1$, 
was first highlighted by Poincar\'e in the \textit{Les M\'ethodes nouvelles de la m\'ecanique c\'eleste} \cite{PoH}; since then,  many authors embraced the investigation of possible quasi-periodic dynamics, giving birth to fruitful fields of research, at the crossroads with dynamical system, geometry and functional analysis.
Among the many results in this field, a breakthrough  has been achieved in 1954 by Kolmogorov \cite{Kolmogorov:1954},
and  the subsequent works of Arnold \cite{Arnold:1963} in 1963 and Moser \cite{Moser:1962} in 1962, opening the way to what is now known as
\emph{KAM theory}.
 The core of their results is 
 that 
 a large measure set of quasi-periodic invariant tori
of a completely  integrable Hamiltonian system  survive 
 sufficiently small  perturbations, under appropriate non-degeneracy conditions.
 A crucial point in proving such a measure-theoretic statement
 is to control resonant/non-resonant interactions $\omega\cdot\ell=\omega_{1}\ell_1+\ldots+\omega_{n}\ell_{n}$ among the frequencies of oscillations
 $\omega\in \R^{n}$ characterising the motion, by imposing quantitative lower bounds to ensure that $\omega\cdot\ell$ is sufficiently  away from zero.
 This is a purely arithmetic problem concerning Diophantine-type inequalities.
 
 \smallskip
 
 In the present paper we shall discuss some aspects of such Diophantine conditions
 in the context of \emph{infinite}-dimensional dynamical systems, close to an elliptic fixed point, 
corresponding to a Hamiltonian partial differential equation. 
More precisely our aim is to focus
 on the frequencies arising from the $1$-parameter family of Klein-Gordon equations of the form
\begin{equation}\label{KG}
\psi_{tt} - \psi_{xx} + \mathtt{m}\psi = \psi^2\,,\qquad x\in\T:=\mathbb{R}/2\pi\Z, \,\; \mathtt{m}\in [1,2]\,,
\end{equation}
(or in presence of a more general non linear term). 
It is know that a good knowledge of the Diophantine properties of the linear frequencies 
of oscillation are fundamental in order to study behaviour of solutions, for long times, of equations like \eqref{KG}.

\smallskip
Let us describe our point of view, coming from dynamical system, in a more general setting. So, let us
consider equations of the form
\begin{equation}\label{irene1}
 \mathfrak{i} u_t  = Lu + N(u)\,,\qquad u=u(x,t)\,,\quad x\in \T:=\R/2\pi\Z
 \end{equation}
where 
\begin{itemize}
\item $u$ belongs to some  Hilbert  subspace of $L^{2}(\T,\C)$;
\item $L$ is a typically unbounded self-adjoint operator with real pure point spectrum $\{\omega_{j}\}_{j\in\Z}\subset\R$ 
\item the nonlinear term $N(u) \sim O(u^{q+1}),\, q\ge 1$.
\end{itemize}
Passing to the Fourier side, i.e.
setting
\[
u(x)=\sum_{j\in\Z}u_je^{\ii jx}\,,\quad (u_{j})_{j\in\Z}\in\ell^{2}(\C)\,,
\]
equation \eqref{irene1} reads as infinitely many linear oscillators coupled by the nonlinearity $N(u)$, i.e. 
\begin{equation}\label{irene2}
 \dot{u}_j =-\ii  \omega_j u_j -\ii  [N(u)]_j\,,\quad j\in\Z\,,
\end{equation}
where the linear frequencies of oscillations are given by the eigenvalues of $L$.
In the following, we shall assume that
\[
[N(u)]_j=\partial_{\bar{u}_{j}}P(u)\,,\qquad P(u)=\int_{\T}F(\sum_{j\in\Z}u_je^{\ii jx})dx\,,
\]
where $F$ is a real analytic function in the neighbourhood of the origin and $F(u)\sim O(u^3)$, so  that we can describe \eqref{irene2}  as an infinite dimensional Hamiltonian system
w.r.t. the symplectic form $\imm \sum_{j\in\Z} du_j\wedge d\bar{u}_j$
and corresponding Hamiltonian 
\begin{equation}\label{carla}
H=\sum_{j\in \Z}\omega_{j}|u_j|^2+P\,,
\quad \quad 
X_H^{(j)} = -\imm \frac{\partial }{\partial{\bar{u}_j}}H(u) \,.
\end{equation}

At the linear level, namely when $P\equiv0$, the solution is
\[
u(t)=\sum_{j\in S}u_{j}(0)e^{\ii\omega_j t} e^{\ii jx}\,,\qquad S:=\{j\in \Z : |u_{j}(0)|\neq 0\}\subseteq\Z\,.
\]
Depending on the arithmetic properties of the frequency vector
$\omega=(\omega_j)_{j\in\Z}\in\mathbb{R}^{\mathbb{Z}}$, the above function $u(t)$ is

\begin{itemize}[leftmargin=*]
\item \emph{almost}-periodic: if 
the set $S$ has cardinality $+\infty$ and one has
\begin{equation}\label{nonrisonanzaS}
\omega\cdot\ell=\sum_{j\in\Z}\omega_{j}\ell_{j}\neq 0\qquad \forall\ell\in\Z^{\Z}\;:\; 0<|\ell|<+\infty\,,
\qquad \ell_{k}=0\,,\;\;\forall k\in S^{c}
\end{equation}


\item \emph{quasi}-periodic:  if
the set $S$ has cardinality $\td <+\infty$ and 
\eqref{nonrisonanzaS} holds true.

\item \emph{periodic}: for any $S$ there exists $T\in \R$ such that $T\omega_{j}\in \Z$ for any $j\in S$.

\end{itemize}

A natural question concerns the possibility of stable behavior, that is whether or not the nonlinear equation \eqref{irene2} possesses solutions 
that remain \emph{close} (in some topology) to the ``linear" ones for long time scales. A stronger notion of stability is whether system \eqref{irene2} still admits periodic, quasi-periodic or almost-periodic solutions, namely if there are invariant subsets 
for the motion, which then remains perpetually confined on such.

 Many partial differential equations admit a form like
 \eqref{irene1}, possibly after appropriate variable's change.
 As pivotal examples we refer to the Nonlinear Schr\"odinger equation, or the celebrated Kortweg- de-Vries equation:
 \begin{align}
 &\ii u_{t}-u_{xx}+\mathtt{m}u=|u|^2u\,,\tag{NLS}
 \\
 &u_{t}-u_{xxx}+uu_{x}=0\tag{KdV}\,,
 \end{align}
 for which it is immediate to see that they have the form \eqref{irene1} with $L=\partial_{xx}-\mathtt{m}$ and $L=\ii \partial_{xxx}$ respectively. 
 In the case of our interest, the  Klein-Gordon model \eqref{KG} can be written as a system of order one
 in terms of the variables $(\psi,v)$ where $v=\dot{\psi}$ and  successively, by introducing suitable complex coordinates, it can be written as 
 in \eqref{irene1}. This is possible using, for instance, 
 the complex variable $(u,\bar{u})$ where 
 \[
 u:=\frac{1}{\sqrt{2}}\big(\Lambda^{\frac{1}{2}}\psi+\ii\Lambda^{-\frac{1}{2}}v \big)\,,
 \qquad \Lambda:=\sqrt{-\partial_{xx}+\mathtt{m}}\,,
 \]
 for which equation \eqref{KG} reads
 \[
 \ii u_{t}=\Lambda u+\frac{1}{\sqrt{2}}\Lambda^{-\frac{1}{2}}\left(\frac{\Lambda^{-\frac{1}{2}}(u+\bar{u})}{\sqrt{2}}\right)^{2}\,.
 \]
  The above examples, together with many others,
  have been extensively studied in the last decades and related  questions around stability 
  and existence of periodic and quasi-periodic solutions have been widely investigated in many contexts and directions.
A guiding argument in this line of thoughts is that for long time scales
 the effect of the nonlinearity becomes non-trivial so that one expects that the dynamics is led by the resonant interactions of the linear frequencies of oscillations. 
 In order to investigate this phenomenon, normal form approaches borrowed from finite dimensional dynamical systems have been successfully extended to infinite dimension and extensively 
 used as an effective tool for proving long time (or almost global) existence of solutions, or existence of special global ones.
 
 It is beyond the purpose of the present paper to provide an overview of stability or \emph{quasi}-periodic KAM theory for PDEs, however,
 we refer for instance to  \cite{CW, Boj, KP, Po2} as the  first 
 works that paved the ground to this theory in infinite dimension, involving phase spaces of functions whose ``space" variable $x$ belongs to some manifold of dimension one.
 From these seminal works, many extensions in different directions raised, depending on the nature of the nonlinearity or the dimension of the space-manifold for example. Without trying to be exhaustive, we address the interested reader to
 \cite{Ku2, LY, BBiP2, KdVAut, FeoGiuPro2020, 
 BBHM, FeoGiu:memo, IonescuSQG}
for 1d equations containing derivatives in the nonlinearity, and to
\cite{B5, EK, EGK, PP3, GengYou2006, BCP}
 for the higher dimensional case, and to the references therein.
 
Regarding the study of almost-periodic solutions instead, only few
examples are known and all of them rely on $1$-space dimensional models depending on infinitely many \textit{external} parameters provided by the presence in the equation of an appropriate convolution or multiplicative potential that is necessary for tuning the infinitely many frequencies of oscillations and get the aforementioned arithmetic Diophantine conditions. However, even in this quite unnatural frame, the problem is hard to handle and the existence of  almost periodic solutions supported on \textit{full} dimensional tori, i.e. whenever $S = \Z$,  has been proved only in functional phase spaces of high regularity, like the analytic or the Gevrey one 
\cite{Bourgain:2005, Poschel:2002, GengXu2013, CongYuan:2020, BMP:AHP}.
However, at the best of our knowledge, a step forward in the direction of lowering the regularity is represented by the recent 
\cite{BMP:weak}.
The question of the existence of such solutions for a \textit{fixed} PDE remains open as well as the possibility of their construction in models with a finite number of natural parameters.
In the following we shall specify to this latter problem and  discuss  non-resonance conditions and the delicate matter of frequency modulation in this degenerate context. 

\smallskip

\vspace{0.5em}
More specifically, \emph{all} the mentioned
 results require very strong non-resonant conditions on the frequencies of oscillations 
 that go under the umbrella of the so-called Diophantine conditions, which now we are going to describe in details.
 More precisely one needs some lower bounds on quantities of the form
 \begin{equation}\label{enneinterazio}
 \omega_{j_{1}}\ell_1+\ldots+ \omega_{j_{n}}\ell_{n}
 \end{equation}
 for any fixed $n\in \N$.
 In the case of a finite number of frequencies those conditions are well understood and the techniques involved to handle them (at least in one dimensional problems) are inherited/adapted from finite dimensional context.
 Concerning an infinite number of frequencies the problem is harder and much less understood.
 
Such lower bounds are achievable only if we
are able to modulate the frequencies $\omega_{j}'$s, which makes somehow necessary to consider an additional linear term in the equation, e.g. a multiplicative or convolution potential.
Note that sometimes the equation may already depend on natural parameters, such as the 
  mass for the Klein-Gordon or Beam equation, or the capillarity for the water waves system for example.
  
  As one can expect, the more parameters one has at disposal the simpler is to impose non-resonance 
  conditions with suitable lower bounds on the quantity \eqref{enneinterazio}.
 
 In the present paper we focus on some aspects of Diophantine
 conditions for infinitely many frequencies $\omega=(\omega_{j})_{j\in\Z}$ 
 in the \emph{degenerate case} of the Klein-Gordon equation where
 \[
\omega_{j}:=\omega_{j}(\mathtt{m}):=\sqrt{j^{2}+\mathtt{m}}\,,\qquad j\in\Z\,,\;\mathtt{m}\in[1,2]\,.
 \]
 The term degenerate refers to the fact that only on parameter is at disposal for the frequencies modulation.
Before analysing the specific properties of the Klein-Gordon frequencies 
we briefly discuss how Diophantine conditions come into play.

 \vspace{0.5em}
 \paragraph{\bf The problem of small divisors.}
 
 As previously mentioned,
 a normal form approach inherited from finite-dimensional dynamical systems  has revealed to be an effective tool to tackle the questions above also in the infinite dimensional context. Consistently with the finite dimensional case, a key idea lies in ``successive linearizations" of the nonlinear problem that consist in straightening the nonlinear flow back to the linear one at the highest possible order, through appropriate diffeomorphisms. Note that, given a sequence of initial data 
 $(I_j)_{j\in\Z} := (|u_j(0)|^2)_{j\in\Z}$ in a chosen phase space $\cP$, 
 the linear flow $u_j(t) = u_j(0) e^{\ii \omega_j t}$ leaves invariant the subset 
 \begin{equation}
 \mathcal{T}_I := \set{u\in\mathcal{P}\, : \, |u_j|^2 = I_j\, \quad \forall j\in\Z}\,. 
 \end{equation}
 Of course, given a Hamiltonian $H$ as in \eqref{carla},
 there is no reason for
 the vector field $X_{P}$ to vanish on $\mathcal{T}_{I}$, therefore its persistence to the (small) nonlinear effects is subordinated to constructing a \emph{close-to-identity} change of coordinates $\Phi$ that conjugates $H $ to the normal form
 \[
 H\circ\Phi=\tD_{\omega}+N\qquad {\rm such\; that} \quad X_{N|_{\mathcal{T}_{I}}}\equiv0\,.
 \] 
where we set $\mathtt{D}_{\omega}$ to be the diagonal term $\mathtt{D}_{\omega} := \sum_j \omega_j|u_j|^2$.
On the other hand, one could aim at a weaker result and look for a diffeomorphism $\Phi$ such that
\[
 H\circ\Phi=\tD_{\omega}+Z+R 
\]
where $Z$ is a polynomial function of $|u_{j}|^{2}$ and $R\sim O(u^{K+2})$ for $K\gg1$ very large.
From that, although there is no invariant torus a priori, one is able to deduce that solutions evolving form data of size $\e$ remain
confined into a ball of radius $2\e$ for times $\sim \e^{-K}$.

\smallskip

 In both cases the map is iteratively constructed as the composition of time-one flows generated by some auxiliary Hamiltonian functions $F$'s. To fix the ideas, let us take a Hamiltonian $F$ and  denote its flow $\Phi_{F}^{\tau}$, $\tau\in[0,1]$.  By assuming that $P\sim O(\e)$ where $0<\e\ll1$ is a small parameter, the Taylor expansion of  
$(H\circ\Phi_{F}^{\tau})_{|\tau=1}$
at zero 
has the form
\[
(H\circ\Phi_{F}^{\tau})_{|\tau=1}=\tD_{\omega}+P+L_{\omega} F+h.o.t.
\]
where
$L_{\omega}$ denotes the Lie
derivative of $F$ along the flow of $\tD_{\omega}$, namely
\begin{equation}\label{def:elleomega}
L_{\omega} F:= \frac{d}{dt}_{|t=0} \Phi_{\mathtt{D}_\omega,t}^* F
\end{equation}
and 
where we denoted by \emph{h.o.t.} the  higher order terms which are at least quadratic in $F\sim O(\e)$.

 So, at each order/step, we can cancel out the terms that obstruct the conjugacy of $H$ to the desired normal form if 
we manage to solve the \textit{homological equation} 
\begin{equation}\label{homologic}
L_\omega F + P = 0\,
\end{equation}
  in some appropriate functional space. 
  \\ In our setting it is natural to consider Hamiltonians that are analytic in some ball around the origin, that is 
given $r>0$ we consider 
$ H : B_r(\cP) \to \R$
such that there exists a point wise  
absolutely convergent power series expansion\footnote{As usual given 
a vector $k\in \Z^\Z$, $|k|:=\sum_{j\in\Z}|k_j|$.}
\begin{equation*}
H(u)  = 
\sum_{\substack{
\al,\bt\in\N^\Z\,, \\2\leq |\al|+|\bt|<\infty} }
\!\!\!
H_{\al,\bt}u^\al \bar u^\bt\,,
\qquad
u^\al:=\prod_{j\in\Z}u_j^{\al_j}\,.
\end{equation*}
  
 Therefore, under the above assumptions, the homological equation yields

\begin{equation}\label{def:adjaction}
L_{\omega} F
= 
\sum_{\substack{
\al,\bt\in\N^\Z\,, \\2\leq |\al|+|\bt|<\infty} } 
-{\rm i}\big(\omega \cdot(\al-\bt)\big)F_{\al,\bt}u^{\al}\bar{u}^{\bt} = - \sum_{\substack{
\al,\bt\in\N^\Z\,, \\2\leq |\al|+|\bt|<\infty} } P_{\al,\bt} u^\al u^\bt\,.
\end{equation}

By identification of coefficients,  it becomes now evident that at least at a formal level the solution $F$ is given by 
$$
F_{\al,\bt} = \frac{P_{\al,\bt}}{{\rm i}\big(\omega \cdot(\al-\bt)\big)}
$$
for any $\al,\bt$ such that $\omega \cdot(\al-\bt) \not\equiv 0$, that is equation \eqref{homologic} is satisfied modulo the kernel of $L_\omega$. \\
In making rigorous the reasoning above we have to deal with  some crucial issues:

\noindent
$\bullet$ the part of $P$ that belongs to the kernel of $L_\omega$, which cannot be eliminated,
contributes to the normal form. Therefore, the description of these terms is fundamental to understand in which way they affect the dynamics;

\noindent
$\bullet$ for proving the convergence of $F$ in the chosen functional space, one needs 
 quantitative lower bounds on the small divisors $|\omega\cdot(\al - \bt)|$
 for any  $\ell=\alpha-\beta$ in $\Z^\Z$ with finite support that belong to the subset
$\Lambda\subseteq\Z^\Z$ of \emph{non-resonant} vectors defined as
\begin{equation}\label{nonresSet}
\Lambda:=\{ \ell\in\Z^\Z\,:\, \omega\cdot\ell\not\equiv0\,,\quad  0<|\ell|<+\infty\}\,;
\end{equation}

\noindent
$\bullet$ 
 for the result to be meaningful, those bounds must be satisfied 
by a positive measure set of frequencies. Therefore, measuring the set of ``good frequencies" becomes a key point.

\smallskip
In the finite dimensional case, i.e. if $\ell\in\Z^d$, $d\geq1$, a classical assumption is to require that  the frequency vector 
$\omega\in\R^d$  is a \emph{dipohantine} vector, i.e. 
there exist $\gamma,\tau>0$ such that 
\begin{equation}\label{diopfinita}
|\omega\cdot\ell|\geq \frac{\gamma}{|\ell|^{\tau}}\,,\quad \forall\, \ell\in\Z^{d}\setminus\{0\}\,.
\end{equation}
It is well known that, if $\tau>d-1$, then  the set of vectors satisfying \eqref{diopfinita} tends to the full measure as $\gamma$ tends to $0$.
Because of this dependence on the dimension $d$ in the diophantine exponent $\tau$ the above condition cannot be extended naively to infinite dimensional $\omega$'s . 
However, under strong assumptions on the asymptotics 
of $\omega_j$, one can impose similar lower bounds on a large measure set of 
$\omega \in\R^\Z$, see \cite{ChierP} for instance.

On the other hand,
the following 
Diophantine-type condition is \emph{uniform} in the dimension of the 
support of $\ell\in \Z^\Z$:
\begin{equation}\label{diofantina}
\mathtt{D}_{\gamma}:=
\Big\{ \omega\in \R^{\Z} : |\omega\cdot\ell|\geq\prod_{n\in\mathbb{Z}} 
\frac{\gamma}{(1+|\ell_{n}|^{2}\langle n\rangle^{2})^{\tau}}\,,\; \forall \ell\in\Lambda: 0 < |\ell| < \infty \Big\}\,,
\end{equation}
where  $\langle n\rangle:=\max\{1,|n|\}$ for any $n\in \mathbb{Z}$, $\gamma, \tau>0$,  
and $\Lambda$ is a suitable \emph{non-resonant} sub-lattice of $\Z^{\Z}$
which, in the applications, depends on the frequencies $\omega$. 
We strongly underline that the above condition, w.r.t. the classical one, 
is tailored for a truly infinite dimensional problem, allowing to construct directly 
an infinite dimensional invariant object avoiding finite dimensional approximations
as it was done in \cite{Poschel:2002, GengXu2013}.

Bourgain introduced for the first time the above condition in \cite{Bourgain:2005} for constructing the first result on almost-periodic solutions for the quintic NLS with convolution potential $V\ast u$. The presence of such potential provides as many parameters as the number of linear frequencies $j^2$ involved, so that  the sub-lattice $\Lambda \equiv \Z^\Z\setminus \set{0}$ and infinitely many $(V_j)_{j\in\Z} \subset [-1/4,1/4]^\Z$ are at disposal for making the set of 
$\omega_j = j^2 + V_j$ in \eqref{diofantina} of measure $1 - O(\gamma)$ (w.r.t. the product probability measure inherited from $[-1/4,1/4]^\Z$ through the map $\omega_j \mapsto \omega_j - j^2 \in [-1/2,1/2]$). \\ 
So far, all the results on the existence of almost-periodic solutions concern PDEs involving infinitely many external parameters, either as Fourier's multiplier like in the convolution potential case, see \cite{CongYuan:2020, BMP:AHP}, 
or as the spectrum of a multiplicative one \cite{Poschel:2002}. 
Whether it is possible to construct such solutions in the case of a fixed (non integrable) 
PDE (i.e. when $V = 0$), which would require frequencies' modulation by moving initial data or 
or a \textit{finite} number of \textit{natural} parameters, is one of the major open questions in the field.

\medskip
So, an intermediate yet fundamental problem becomes whether 
one can use a number of parameters strictly less than the number of frequencies $\omega_{j}$'s for fulfilling the diophantine
 conditions required in the construction of quasi/almost periodic solutions.
 So far, only the quasi-periodic case has been successfully tackled, 
by means of the so called \emph{degenerate KAM theory}, we refer the reader to \cite{russmann2001, BambuBertiMagi:2011} for example.
We remark that this degenerate case naturally arises from several physical models in which the linear frequencies depend on some \emph{internal} physical parameter.
To be more concrete we mention the Beam , the Klein-Gordon and the gravity-capillary Water Waves equations
whose dispersion relations are respectively given by 
\begin{itemize}
\item \emph{Beam}: \;\;$\displaystyle{\Z\ni j\mapsto \omega_{j}:=\sqrt{|j|^{4}+\mathtt{m}}}$;

\item \emph{Klein-Gordon}: \;\;$\displaystyle{\Z\ni j\mapsto \omega_{j}:=\sqrt{|j|^{2}+\mathtt{m}}}$;

\item \emph{Gravity-capillary Water Waves}: \;\;$\displaystyle{\Z\ni j\mapsto \omega_{j}:=\sqrt{\kappa|j|^{3}+g|j|}}$,
\end{itemize}
where $\mathtt{m}>0$ is the \emph{mass} while $\kappa,g>0$ are respectively the capillarity of the fluid and the gravity.

A common point in these examples is that  there are two main obstruction 
in proving that the set $\tD_{\gamma} $ in \eqref{diofantina}
has large measure when $\tau >0$ is a pure number and $\Lambda=\Z^{\Z}\setminus\{0\}$:

\smallskip
\noindent
$(i)$ because of the parity of the dispersion law the set of resonant $\ell=\alpha-\beta$, $(\alpha,\beta)\in\N^{\Z}\times\N^{\Z}$
is not reduced only to $\alpha\equiv\beta$.
Indeed one can show that the \emph{resonant} subset $\Lambda^{c}$ necessarily contains the set
\begin{equation}\label{def:resonantset}
\mathtt{R}:=\left\{
\begin{aligned}
\ell=\alpha-\beta\,,&\;(\alpha,\beta)\in\N^{\Z}\times\N^{\Z} \;:\; 0<|\alpha|+|\beta|<+\infty\,,
\\&
\sum_{j\in\Z}j\alpha_{j}-j\beta_{j}=0\,,\quad
\alpha_{j}=\beta_{j} \;\lor\; \alpha_{j}=\beta_{-j}\; \forall\; j\in\Z
\end{aligned}
\right\}\,.
\end{equation}
Actually one is able to prove (this will be discussed in details later in the Klein-Gordon example) that indeed
the resonant set is exactly  $\mathtt{R}$.

In the favourable case of the NLS with convolution potential (for instance), 
where the dispersion relation is 
\begin{equation}\label{freqNLS}
\Z\ni j \to \omega_j := j^2 + V_j\quad (V_j)_{j\in\Z}\subset\ell^{\infty}
\end{equation}
one  can employ a potential satisfying $V_{j}\neq V_{-j}$ in order to prove that the resonant set $\mathtt{R}$
reduces to $\alpha\equiv\beta$.

\smallskip

\smallskip
\noindent
$(ii)$ Concerning the measure of the corresponding set $\mathtt{D}_{\gamma}$
it seems, at the moment, out of reach to obtain positive measure for an exponent $\tau>0$
independent of the \emph{support} of $\ell$. Indeed, due to the degenerate setting, it is not possible to 
estimate the sub-levels of $\omega\cdot\ell$ just providing lower bounds on its first derivative\footnote{This is actually what can be done in the case of \eqref{freqNLS} where many parameters are at disposal.}.
On the contrary one needs different results, 
involving a certain (large) number of derivatives, which provides slightly worst measure estimates
on sub-levels of $C^k$-functions. In addition to that such estimates depend on the support of $\ell$.
We refer to \cite{Eliasson:cetraro} and to section \ref{sec:debole} of the present paper.

\vspace{0.5em}
\noindent
{\bf Main result.} We are now in position to state our main result.
Let us consider 
\begin{equation}\label{dispLawWave}
\begin{aligned}
\omega&:=\omega(\mathtt{m}):=(\omega_{j})_{j\in \mathbb{Z}}\in \mathbb{R}^{\mathbb{Z}}\,,
\\
\omega_{j}&:=\omega_{j}(\mathtt{m}):
=\sqrt{|j|^{2}+\mathtt{m}}\,,\qquad j\in \mathbb{Z}\,,\qquad \mathtt{m}\in[1,2]\,,
\end{aligned}
\end{equation}
%
%
%
%
and  the set of non-resonant indexes 
\begin{equation}\label{restrizioni indici}
\Lambda:=\big\{\ell\in \mathbb{Z}^{\mathbb{Z}}\; : \; 
\ell:=\alpha-\beta\,,\; \forall (\alpha,\beta)\in \mathtt{R}^{c}\big\}\,,
\end{equation}
where $\mathtt{R}$ is in \eqref{def:resonantset}.
Moreover, given a vector $\ell:=(\ell_i)_{i\in \mathbb{Z}}\in \Lambda$
consider the set
\begin{equation}\label{insiemeAA}
\mathcal{A}(\ell):=\{i\in \mathbb{Z} \,:\, \ell_i\neq0\}
\end{equation}
 and
define the map
\begin{equation}\label{polloarrosto}
\ell\mapsto \mathtt{d}:=\mathtt{d}(\ell)\in \mathbb{N}
\end{equation}
where $\mathtt{d}(\ell):=\#\mathcal{A}(\ell)$. We call $\mathtt{d}(\ell)$ the \emph{support} of $\ell$,
i.e. the number of components of $\ell$ which are different form zero.
We need the following definition.
\begin{definition}
Consider a vector $v=\pa{v_i}_{i\in \Z}$,  $v_i\in \N$, $|v|<\infty$. 
We define $m=m(v)=(m_{j})_{j\in\mathbb{N}}$ as the reordering of the elements of the set
\[
\set{j\neq 0 \,,\quad \mbox{repeated}\quad  \abs{v_j} \;\mbox{times}}\,,
\]
where $D<\infty$ is its cardinality, such that
$|m_1|\ge |m_2|\ge \dots\geq |m_D|\ge 1$. 	
\end{definition}

\begin{theorem}{\bf (Measure estimates for the Klein-Gordon).}\label{thm:mainVERO}
There exists a positive measure set $\mathfrak{Q}\subseteq [1,2]$
such that
for any $\mathtt{m}\in\mathfrak{Q}$, 
 the vector $\omega(\mathtt{m})$ 
satisfies the following: for any 
$\ell=\alpha-\beta$, $(\alpha,\beta)\in \Lambda$ 
one has 
\begin{equation}\label{goodsmalldivTeorema}
|\omega\cdot\ell|\geq  \gamma^{{\tau^2(\ell)}}
\frac{1}{(\mathtt{C}(|\alpha|+|\beta|))^{6\tau(\ell)}}
\prod_{\substack{n\in\mathbb{Z} \\ n\neq m_1(\ell),m_2(\ell)}} 
\frac{1}{(1+|\ell_{n}|^{2}\langle n\rangle^{2})^{{4\tau^2(\ell)}}}\,,
\end{equation}
where $\tau(\ell):=\td(\ell)(\td(\ell)+4)$, 
and $\mathtt{C}$ 
is a 
positive pure constant  large enough.
Moreover,  there exists a positive constant $\mathtt{c}$ such that
\[
\meas([1,2]\setminus \mathfrak{Q})\leq \mathtt{c}\gamma\,.
\]
\end{theorem}

Some comments are in order.

\begin{enumerate}
\item As already explained the exponents in the r.h.s. of  \eqref{goodsmalldivTeorema}
depends on $\td(\ell)$, differently from \eqref{diofantina}. This dependence, which seems to us to be unavoidable,
is the major obstruction to the possibility of constructing \emph{almost}-periodic solutions
following the ideas in \cite{BMP:AHP}.

\item Differently from \eqref{diofantina}, the bound \eqref{goodsmalldivTeorema} 
does not depend on the highest indexes $m_1(\ell),m_2(\ell)$.
Note that in the case of super-linear dispersion relations (such as NLS or Beam where $\omega_{j}\sim j^2$)
the estimates 
\eqref{diofantina} can be improved similarly. An idea like this has 
been implemented for instance in \cite{BMP:AHP} and 
\cite{FeoMass:Beam}. The asymptotically linear dispersion law of the Klein-Gordon make this step non-trivial.
See section \ref{sec:improved} for the technical details.

\item The improvement in the  lower bound \eqref{goodsmalldivTeorema} has been achieved also 
in \cite{CongYuan:2020}, in the case of the Klein-Gordon equation with a convolution potential. 
We obtain a similar result with just one parameter modulating the frequencies,
at the price of the presence of $\tau(\ell)$ (depending on $\ell$) in the exponent.
\end{enumerate}

The proof of the Theorem above consists in two steps.

We start by proving that  a first order Melnikov condition holds with a weaker lower bound involving also the first two highest indexes 
$m_{1}(\ell),m_2(\ell)$. This is the content of Proposition \ref{thm:main}.
Then, in subsection \ref{sec:improved},
we improve such lower bound exploiting the asymptotic, as $j\to\infty$, of the frequencies $\omega_{j}$.
This step strongly relies on the one dimensional setting.

Although a KAM type result for almost periodic solutions seems to be out of reach,   
 as an application 
 of Theorem \ref{thm:mainVERO} we show how to use the above estimates to solve a 
 homological equation 
\eqref{def:adjaction} in Gevrey category. More precisely in Theorem \ref{shulalemma} we are able to provide
quite sharp estimates on $L_{\omega}^{-1}P$ where $P$ is a homogeneous polynomial of any degree
$\mathtt{N}+2$. Our estimates depend on $\mathtt{N}$, and  tend to infinity as $\mathtt{N}\to \infty$,
due to the presence of $\tau(\ell)$ and because the support $\td(\ell)\sim \mathtt{N}^{2}$. 
This prevents the implementation of a KAM scheme as in \cite{BMP:AHP}. 
However, a long time stability result is likely
to be achieved, in line with \cite{FeoMass:Beam}.

We conclude this introduction by giving some key ideas of the proof of Theorem \ref{shulalemma}.

In section \ref{sec:classe} we introduce Gevrey-type norms on Hamiltonians.
It turns out, exploiting that formalism, that  the norms of the solution $L_{\omega}^{-1}P$
is subordinated to bound form above the quantity, $\theta\in(0,1)$,
\[
J_0 =  
\sup_{\substack{j\in\Z,\,  (\al,\bt)\in\Lambda \\ 
\al_j+\bt_j\neq 0  \\
 |\al-\bt|\leq \tN+2}}
\frac{e^{-\s\pa{\sum_i\jap{i}^{\theta} (\al_i+\bt_i) -2\jap{j}^{\theta}}}}{\abs{\omega\cdot{\pa{\al - \bt}}}}\,.
\]
Here $\s>0$ denotes the usual ``loss'' of regularity of $P$. 
The crucial point is that the exponent is negative, and in particular  
\[
\sum_i\jap{i}^{\theta} (\al_i+\bt_i) -2\jap{j}^{\theta}\geq {\sum_{i\neq m_1,m_2}\langle i\rangle^{\theta}|\alpha_i-\beta_i|} \,,
\]
see Lemma  \ref{lem:constance2SE}.
This shows why it is fundamental to have a lower bound in \eqref{goodsmalldivTeorema} 
independent of the highest indexes 
$m_{1}(\ell), m_{2}(\ell)$. We remark that, thanks to the superlinear dispersion, 
in \cite{BMP:AHP} and \cite{FeoMass:Beam} the authors proves 
an improved version of our Lemma \ref{lem:constance2SE} showing that
\[
\sum_{i\in\Z}\jap{i}^{\theta} (\al_i+\bt_i) -2\jap{j}^{\theta}\geq {\sum_{i\in\Z}\langle i\rangle^{\theta}|\alpha_i-\beta_i|} \,.
\]
Of course this allows to impose a weaker diophantine condition  like \eqref{diofantina}.

\section{Small divisors}\label{sec:smalldiv}
Here we show some type of  lower bounds 
that one can impose 
on $\omega\cdot\ell$, $\ell\in \Z^{\Z}$, where $\omega$ is in \eqref{dispLawWave}.

\subsection{A weak-diophantine condition for the Klein-Gordon}\label{sec:debole}
In this section we prove the following proposition.

\begin{proposition}\label{thm:main}
There exists a positive measure set $\mathfrak{Q}\subseteq [1,2]$
such that
for any $\mathtt{m}\in\mathfrak{Q}$, the vector $\omega(\mathtt{m})$ defined in \eqref{dispLawWave}
belongs to the diophantine set of frequencies 
\begin{equation}\label{diofSetPRIMA}
{\mathtt{P}}_{\gamma}:=
\Big\{ \omega\in \R^{\Z} : |\omega\cdot\ell+p|\geq\gamma^{\mathtt{d}(\ell)}\prod_{n\in\mathbb{Z}} 
\frac{1}{(1+|\ell_{n}|^{2}\langle n\rangle^{2})^{{\tau}}}\,,\; \tau:=\mathtt{d}(\ell)(\mathtt{d}(\ell)+4)\,,\;\;
\forall\ell\in \Lambda\,,\; p\in\Z \Big\}\,,
\end{equation}
Moreover,  there exists a positive constant $\mathtt{C}$ such that
\[
\meas([1,2]\setminus \mathfrak{Q})\leq \mathtt{C}\gamma\,.
\]
\end{proposition}

The proof of the Theorem  involves several argument which will be discussed below.
First of all let us define the quantity (see \eqref{mass-momindici})
\begin{equation}\label{smallDiv}
{\bf \Psi}(\mathtt{m},\ell,p):=\psi(\mathtt{m},\ell)+a\,,\qquad \psi(\mathtt{m},\ell):=\omega\cdot\ell\,, \qquad \forall \ell\in \mathcal{M}\,,\;\;\; a\in \Z\,,
\end{equation}
and recall that we shall provide lower bounds on $\psi(\omega, \ell)$
only for $\ell$ belonging to the set $\Lambda$ in \eqref{restrizioni indici}.
Moreover, according to the notation \eqref{polloarrosto} and \eqref{insiemeAA}, 
we can write the function in \eqref{smallDiv} as
\begin{equation}\label{pizzapomo}
\psi(\mathtt{m},\ell)=\sum_{i=1}^{\mathtt{d}}\ell_{j_i}\omega_{j_i}\,,
\qquad j_i\in\mathcal{A}(\ell)\subset\mathbb{Z}\,.
\end{equation}

\vspace{0.5em}
\noindent
{\bf Estimates of a single ``bad set''.}
We consider, for any fixed $\ell\in \Lambda$  and $\eta>0$,
we define the ``bad set'' of parameters 
\begin{equation}\label{diofSet44}
\mathcal{B}(\ell,	a):=
\Big\{
\mathtt{m}\in [1,2] \, :\, |\omega\cdot\ell+a|
\leq
\gamma^{\mathtt{d}}
\prod_{n\in\mathbb{Z}}\frac{1}{(1+|\ell_{n}|^{2}\langle n\rangle^{2})^{\tau}} 
\Big\}
\end{equation}
with $\tau$ as in \eqref{diofSetPRIMA}. In the following we show that  Lebesgue measure of $\mathcal{B}(\ell,a)$ is bounded by
$\gamma g(\ell)$ where $g(\ell)$ decays in $\ell\in \Lambda$. 
The ideas  involved are quite standard, see for instance \cite{Bambusi-Grebert:2006}.
The purpose of this subsection is to adapt such ideas to our non classical Diophantine condition.

We have the following.
\begin{lemma}\label{lem:vander}
For any $\ell\in \Lambda$ there exists $1\leq k\leq \mathtt{d}(\ell)$ such that
\begin{equation}\label{albero2}
|\partial_{\mathtt{m}}^{k}\psi(\mathtt{m},\ell)|
{\geq\prod_{j\in\mathcal{A}(\ell)}\frac{1}{(1+|\ell_{j}|^{2}\langle j\rangle^{2})^{{\mathtt{d}(\ell)}+1}}}
\,.
\end{equation}
\end{lemma}

\begin{proof}
To lighten the notation we shall write $\mathtt{d}$ instead of $\mathtt{d}(\ell)$.
{Moreover for any fixed $\ell\in \Lambda$ (recall \eqref{insiemeAA})
we shall write $\mathcal{A}(\ell)\equiv\{j_{1},\ldots,j_{\td}\}$ for some $j_{i}\in \Z$, $i=1,\ldots, \td$.}
In this way, after a reordering of the indexes we can write 
$\ell = (\bar\ell,0),$  where $\bar\ell = (\ell_{j_1},\ldots,\ell_{j_{\td}})$. 
Without loss of generality, 
we can always assume that the vector $\bar{\ell}$ satisfies 
\begin{equation}\label{Nosuperact}
j_i\neq-j_k\,,\qquad \forall\; j,k=1,\ldots,\mathtt{d}\,.
\end{equation}
Indeed,  the d-pla $({j_1},\ldots,{j_{\td}})$ can be written as
\[
(k_1,\ldots, k_{p}, q_1,-q_1, q_2, -q_2\, \ldots, q_{r} , - q_{r})\,,\qquad 0\leq p\leq \mathtt{d}
\]
for some $0\leq p\leq \mathtt{d}$ and $p+2r=\mathtt{d}$, 
where $k_i$, $i=1,\ldots, p$ satisfy \eqref{Nosuperact}. 
The small divisors has the form
\[
\omega\cdot \bar{\ell}=\sum_{i=1}^{\mathtt{d}}\omega_{j_i}\ell_{j_i}=
\sum_{i=1}^{p}\omega_{k_i}\ell_{k_i}+\sum_{i=1}^{r} \omega_{q_i}(\ell_{q_i}+\ell_{-q_i})\,.
\]
Hence we can define
\[
\tilde{\ell}=(\tilde{\ell}_{k_1},\ldots, \tilde{\ell}_{k_p}, \tilde{\ell}_{q_1}, \ldots, \tilde{\ell}_{q_r} )\,,
\qquad {\rm where} \quad \left\{\begin{aligned}
&\tilde{\ell}_{k_i}=\ell_{k_i}\,,\quad i=1,\ldots p\,,
\\
& \tilde{\ell}_{q_i}=\ell_{q_i}+\ell_{-q_i}\,,\quad i=1,\ldots, r\,.
\end{aligned}\right.
\]
Since $\ell\in \Lambda$
it is not possible that at the same time $p=0$ and $\ell_{q_i}+\ell_{-q_i}=0$
 for any $i=1,\ldots, r$. Otherwise $\ell$ is a resonant vector (recall \eqref{def:resonantset}). 
 As a consequence up to reducing the length of $\tilde{\ell}$ to $\tilde{\td}=\td(\tilde{\ell})\leq \td(\ell) $
 (by eliminating the components for which $\ell_{q_i}+\ell_{-q_i}=0$), 
 we have obtained a vector satisfying condition \eqref{Nosuperact}
 with $\tilde{\td}\leq \td$.

Hence from now on we consider $\ell\in \Lambda$  with $\td(\ell)=\td$ and satisfying \eqref{Nosuperact}.
\noindent
Notice that, for any $k\geq1$, 
\begin{equation}\label{albero3}
\partial_{\mathtt{m}}^{k}\psi(\mathtt{m},\ell)
=
\sum_{i=1}^{\mathtt{d}}\ell_{j_i}\partial_{\mathtt{m}}^{k}\omega_{j_i}
=
\Gamma(k)\sum_{i=1}^{\mathtt{d}}\ell_{j_i}(\omega_{j_i})^{1-2k}\,,
\end{equation}
where
\[
\Gamma(1) = \frac12,\quad \Gamma(k):=\frac{(-1)^{k+1}}{2^{k}}(2k-3)!! \quad k\geq 2 \,.
\]
Let us define $\mathtt{a}:=(\mathtt{a}_i)_{i=1,\ldots,\mathtt{d}}\in \mathbb{R}^{\mathtt{d}}$
as
\begin{equation}\label{albero4}
\partial_{\mathtt{m}}^{k}\psi(\mathtt{m},\ell)
=
\mathtt{a}_{k}\,,\qquad k=1,\ldots,d\mathbb{d}\,.
\end{equation}
Our aim is to prove that there is at least one component of the vector $\mathtt{a}$
satisfying the bound \eqref{albero2}.
In view of \eqref{albero3} we rewrite \eqref{albero4} as
\begin{equation}\label{albero6}
\Gamma M O \ell=\mathtt{a}\,,
\end{equation}
where the $\mathtt{d}\times\mathtt{d}$ matrices are defined as 
\begin{equation}\label{albero10}
\begin{aligned}
\Gamma&:=\left(
\begin{matrix}
\Gamma(1) &\ldots & \ldots&0 \\
0 & \Gamma(2) & \ldots &\vdots\\
\vdots &\ldots & \ddots & \vdots \\
0&\ldots &\ldots & \Gamma(\mathtt{d})
\end{matrix}
\right)\,,
\qquad
O:=\left(
\begin{matrix}
\omega_{j_1}^{-1} &\ldots & \ldots&0 \\
0 & \omega_{j_2}^{-1} & \ldots &\vdots\\
\vdots &\ldots & \ddots & \vdots \\
0&\ldots &\ldots & \omega_{j_\mathtt{d}}^{-1}
\end{matrix}
\right)\,,
\\
M&:=\left(
\begin{matrix}
1 &\ldots & \ldots&1 \\
\omega_{j_1}^{-2\cdot 1} & \ldots & \ldots &\omega_{j_{\mathtt{d}}}^{-2\cdot 1}\\
\vdots &\ldots & \ldots & \vdots \\
\omega_{j_{1}}^{-2\cdot (\mathtt{d}-1)}&\ldots &\ldots &\omega_{j_{\mathtt{d}}}^{-2\cdot(\mathtt{d}-1)}
\end{matrix}
\right)\,.
\end{aligned}
\end{equation}
Notice that the matrix $M$ is a Vandermonde matrix.
Moreover using that $\ell\in \Lambda$ and that \eqref{Nosuperact} holds, 
its  determinant is given by
\[
{\rm det}(M)=\prod_{i\neq k}(\omega_{j_i}^{-2}-\omega_{j_k}^{-2})
\neq 0\,,
\]
so that the matrix $M$ is invertible.
It is also easy to check that
\[
\max_{i,k=1,\ldots,\mathtt{d}}|(M^{-1})_{i}^{k}|\leq (\mathtt{d}-1)!
\prod_{i\neq k}\frac{\omega_{j_i}^{2}\omega_{j_k}^{2}}{\omega_{j_i}^{2}-\omega_{j_k}^{2}}\,
{\lesssim} \,2^{-\mathtt{d}}\mathtt{d}^{-1}\big(\prod_{i=1}^{\mathtt{d}}\omega_{j_i}\big)^{\mathtt{d}}\sim
2^{-\mathtt{d}}\mathtt{d}^{-1}\big(\prod_{i=1}^{\mathtt{d}}\jap{j_i}\big)^{\mathtt{d}}\,.
\]
Recalling \eqref{albero10} we note
\[
\max_{i=1,\ldots,\mathtt{d}}|(\Gamma^{-1})_{i}^{i}|\leq 2^{\mathtt{d}}\,,\quad
\max_{i=1, \ldots,\mathtt{d}}|(O^{-1})_{i}^{i}|\leq \max_{i=1,\ldots,\td}\jap{j_i}\,.
\]
Therefore
\begin{equation}\label{albero7}
\max_{i,k=1,\ldots,\mathtt{d}}|\big((\Gamma MO)^{-1})_{i}^{k}|\lesssim
\mathtt{d}^{-1}
\big(\prod_{i=1}^{\mathtt{d}}\jap{j_i}\big)^{\mathtt{d}+1}\lesssim 
\mathtt{d}^{-1}
\Big(\prod_{i=1}^{\mathtt{d}}
\frac{1}{(1+|\ell_{j_i}|^{2}\langle j_i\rangle^{2})}\Big)^{-\mathtt{d}-1}\,.
\end{equation}
Since by \eqref{albero6}, we have $\ell=(\Gamma MO)^{-1}\mathtt{a}$,
we deduce
\[
1\leq |\ell|\lesssim \mathtt{d}
\max_{i,k=1,\ldots,\mathtt{d}}|\big((\Gamma MO)^{-1})_{i}^{k}|\|\mathtt{a}\|_{\ell^{\infty}}\,,
\]
which, together with \eqref{albero7}, implies the bound \eqref{albero2}.
\end{proof}

Now we need the following result 
(see for example Lemma B.1 \cite{Eliasson:cetraro} ):
\begin{lemma}\label{v.112}
Let  $\mathfrak{g}(x)$ be a $C^{n+1}$-smooth 
function on the segment $[1,2] $ such that 
\[
|\mathfrak{g}'|_{C^n} =\beta \quad {\rm and}\qquad  
\max_{1\le k\le n}\min_x|\partial^k \mathfrak{g}(x)|=\sigma\,.
\]
Then  one has
\[
\meas(\{x\mid |\mathfrak{g}(x)|\leq\rho\} )\leq 
C_n \left(\beta \s^{-1}+1\right) 
(\rho \s^{-1})^{1/n}\,.
\]
\end{lemma}

 Thanks to Lemma \ref{lem:vander} we shall apply Lemma \ref{v.112}
with $n=\mathtt{d}$ and
{
\[
\begin{aligned}
&
\s\geq \prod_{j\in \mathcal{A}(\ell)}
\frac{1}{(1+|\ell_{j}|^{2}\langle j\rangle^{2})^{\mathtt{d}(\ell)+1}}\,,
\\&
\rho=\gamma^{\mathtt{d}(\ell)}\prod_{n\in\mathbb{Z}}
\frac{1}{(1+|\ell_{n}|^{2}\langle n\rangle^{2})^{\tau(\ell)}}
=\gamma^{\mathtt{d}(\ell)}\prod_{j\in\mathcal{A}(\ell)}
\frac{1}{(1+|\ell_{j}|^{2}\langle j\rangle^{2})^{\tau(\ell)}}
\\&
\beta\leq\mathtt{d}(\ell)!\lesssim \prod_{j\in\mathcal{A}(\ell)}(1+|\ell_{j}|^{2}\langle j\rangle^{2})\,.
\end{aligned}
\]
}
Therefore we obtain
{
\begin{equation}\label{misuraBad}
\meas(\mathcal{B}(\ell,a))\lesssim\gamma
\Big(\prod_{j\in\mathcal{A}(\ell)}
\frac{1}{1+|\ell_{j}|^{2}\langle j\rangle^{2}}\Big)^{\frac{\tau(\ell)}{\mathtt{d}(\ell)}-\mathtt{d}(\ell)-2}\,.
\end{equation}
}

\vspace{0.5em}
\noindent
{\bf Summability.} 
Recalling \eqref{diofSet44} we define the set 
\[
\mathcal{B}:=\bigcup_{\ell\in\Lambda,a\in \Z}\mathcal{B}(\ell,a)\,.
\]
In view of  \eqref{diofSetPRIMA} we have that Theorem \ref{thm:main} follows
by setting $\mathfrak{Q}:=[1,2]\setminus\mathcal{B}$ and provided that 
$\mathcal{B}$ satisfies the estimate
\[
|\mathcal{B}|\lesssim\gamma\,.
\]
In order to prove this we first need the following.
\begin{lemma}\label{lem:taglio}
For any fixed $\ell\in \Lambda$
we have that $\mathcal{B}(\ell,a)\neq \emptyset$ implies that 
\begin{equation}\label{caldo21}
|a|\leq\mathtt{k}(\ell):=2 \prod_{i\in\Z}(1+|\ell_i|\langle i\rangle^2)\,.
\end{equation}
\end{lemma}
\begin{proof}
First of all, by \eqref{dispLawWave}, we note that $|\omega_{j}|\leq |j|+1$ for any $j\in \Z$. Then, for any $\ell\in \Lambda$, we have 
\begin{equation}\label{caldo10}
|\omega\cdot\ell|\leq \prod_{i\in\Z}(1+|\ell_i|\langle i\rangle^2)\,.
\end{equation}
If $|a|\leq |\omega\cdot\ell|$ the thesis follows trivially. On the other hand assume that $|\omega\cdot\ell|<|a|$. 
Moreover, the assumption $\mathcal{B}(\ell,a)\neq \emptyset$ implies that there exists $\mathtt{m}\in[1,2]$ such that
\[
|\omega\cdot\ell+a|\leq 
\gamma^{\mathtt{d}}\prod_{n\in\mathbb{Z}}\frac{1}{(1+|\ell_{n}|^{2}\langle n\rangle^{2})^{\tau}} \leq 1\,.
\]
By the triangular inequality we deduce
\[
1\geq |\omega\cdot\ell+a|\geq \big||a|-|\omega\cdot\ell|\big|=|a|-|\omega\cdot\ell|
\qquad \Rightarrow\quad
|a|\leq 1+|\omega\cdot\ell|\,.
\]
The latter bound, together with \eqref{caldo10}, implies the thesis.
\end{proof}
\begin{remark}\label{fender}
Recalling the notation introduced in \eqref{insiemeAA}-\eqref{polloarrosto} we shall write,
see \eqref{caldo21},
\[
\mathtt{k}(\ell)=2\prod_{j\in\mathcal{A}(\ell)}(1+|\ell_{j}|\langle j\rangle^{2})
\]
\end{remark}

\begin{proof}[{\bf Proof of Proposition \ref{thm:main}}]
In view of Lemma \ref{lem:taglio} we have
\[
\mathcal{B}:=\bigcup_{\ell\in\Lambda,a\in \Z}\mathcal{B}(\ell,a)\subseteq\bigcup_{\ell\in\Lambda,|a|\leq \mathtt{k}(\ell)}\mathcal{B}(\ell,a)\,,
\]
so that (recall also Remark \ref{fender})
\[
\begin{aligned}
|\mathcal{B}|&\lesssim \sum_{\ell\in\Lambda,|a|\leq \mathtt{k}(\ell)}|\mathcal{B}(\ell,a)|\lesssim
\sum_{\ell\in\Lambda}\mathtt{k}(\ell)|\mathcal{B}(\ell,a)|
\stackrel{\eqref{misuraBad},\eqref{caldo21}}{\lesssim}
\sum_{\ell\in \Lambda}
\gamma
\Big(\prod_{j\in \mathcal{A}(\ell)}
\frac{1}{1+|\ell_{j}|^{2}\langle j\rangle^{2}}\Big)^{\frac{\tau(\ell)}{\mathtt{d}(\ell)}-\mathtt{d}(\ell)-3 }
\end{aligned}
\]
Recalling the choice $\tau(\ell)=\td(\ell)(\mathtt{d}(\ell)+4)$ in \eqref{diofSetPRIMA}, we deduce that
$|\mathcal{B}|\lesssim\gamma$ provided that
\[
\sum_{\ell\in \Lambda}
\Big(\prod_{j\in \mathcal{A}(\ell)}
\frac{1}{1+|\ell_{j}|^{2}\langle j\rangle^{2}}\Big)\simeq
\sum_{\ell\in\Lambda}\Big(\prod_{j\in\Z}
\frac{1}{1+|\ell_{j}|^{2}\langle j\rangle^{2}}\Big)\lesssim1\,.
\]
The latter bound follows reasoning exactly as in the proof of Lemma $4.1$ in \cite{BMP:CMP}.
This concludes the proof.
\end{proof}

\begin{remark}\label{rmk:ker2}
It is easy to note that if $(\al,\bt)\in \mathtt{R}$ and $\ell=\al-\bt$, the condition in  \eqref{def:resonantset}
implies that 
\[
\ell_{j}+\ell_{-j}=\al_j-\bt_j+\al_{-j}-\bt_{-j}\equiv0\,,\qquad \forall\, j\in\Z\,.
\]
Therefore  for any $(\al,\bt)\in \mathtt{R}$
and $\ell=\al-\bt$ one has that $\omega\cdot\ell\equiv0$ is identically zero  for $\mathtt{m} \in[1,2]$.
On the other hand, by Theorem \ref{thm:main}, for any $\omega\in \mathtt{P}_{\gamma}$
 one has $\omega\cdot \ell\neq0$ for any $\ell\in\Lambda$. 
\end{remark}

\subsection{Improved estimates for the Klein-Gordon equation}\label{sec:improved}
In this section we improve the estimates obtained in subsection \ref{sec:debole}
and we conclude the proof of our main Theorem \ref{thm:mainVERO}.
Define the set 
\begin{equation}\label{mass-momindici}
\mathcal{M}:=
\left\{
(\alpha,\beta)\in \mathbb{N}^{\Z}\times\N^{\Z} : \pi(\al - \bt):= \sum_{j\in\Z}j\pa{\al_j - \bt_j}= 0, 
|\alpha|+|\beta|<\infty
\right\}\,.
\end{equation}
In the following it will be convenient to use the following way of reordering of the indexes
$j\in \Z$ appearing in the Hamiltonian \eqref{HamPower}.

\begin{definition}\label{n star}
Consider a vector $v=\pa{v_i}_{i\in \Z}$  $v_i\in \N$, $|v|<\infty$. 
	
\noindent
$(i)$ We denote by $\na=\na(v)$ the vector $\pa{\na_l}_{l\in I}$ 
(where $I\subset \N$ is finite)  
which is the decreasing rearrangement of
\[
\{\N\ni h> 1\;\; \mbox{ repeated}\; v_h + v_{-h}\; \mbox{times} \} 
\cup 
\set{ 1\;\; \mbox{ repeated}\; v_1 + v_{-1} + v_0\; \mbox{times}  }
\]
	
\noindent
$(ii)$ Define the vector $m=m(v)$ as the reordering of the elements of the set
\[
\set{j\neq 0 \,,\quad \mbox{repeated}\quad  \abs{v_j} \;\mbox{times}}\,,
\]
where $D<\infty$ is its cardinality, such that
$|m_1|\ge |m_2|\ge \dots\geq |m_D|\ge 1$. 	
\end{definition}


\medskip
 Given $\al\neq\bt\in\N^\Z,$ with $|\al|+|\bt|<\infty$
 we consider $m=m(\al-\bt)$ and $\na=\na(\al+\bt)$.	
If we denote by $D$ the cardinality of $m$ and $N$ the one of $\na$ we have 
\begin{align}
D+\al_0+\bt_0&\le N\,, \label{cappella}
\\
(|m_1|,\dots,|m_D|,\underbrace{1,\;\dots \;,1}_{N-D\;\rm{times}} )\, 
&\leq\,
\pa{\na_1,\dots \na_N}\,.\label{abbacchio}
\end{align}
Set $\s_l= {\rm sign}(\al_{m_l}-\bt_{m_l})$ and note that
for every function $g$ defined on $\Z$ we have that
\begin{equation}\label{pula2}
\begin{aligned}
\sum_{i\in\Z} g(i) |\al_i-\bt_i|
&=
g(0)|\al_0-\bt_0|+
\sum_{l\geq 1} g(m_l)\,,
\\
\sum_{i\in\Z} g(i) (\al_i-\bt_i)
&=
g(0)(\al_0-\bt_0)+
\sum_{l\geq 1} \s_l g(m_l)\,.
\end{aligned}
\end{equation}
In particular, in view of \eqref{pula2}, we shall write (recall that $\ell=\alpha-\beta$)
\begin{align}
0=\pi&
=\sum_{i\in\Z}(\alpha_{i}-\beta_{i})i=\sum_{l\geq1}\s_{l}m_{l}=\s_{1}m_1+\s_{2}m_2+\sum_{l\geq3}\s_{l}m_{l}
\label{espandomomento}
\end{align}
and, setting $\widetilde{\ell}=\ell-\s_{1}e_{m_{1}}-\s_{2}e_{m_{2}}$,
\begin{align}
\omega\cdot\ell&=\sum_{i\in\Z}(\alpha_{i}-\beta_{i})\omega_{i}
={(\alpha_0 - \beta_0)\omega_0} + \sum_{l\geq1}\s_{l}\omega_{m_l}
=\s_{1}\omega_{m_1}+\s_2\omega_{m_2}+\omega\cdot\widetilde{\ell}\,,
\,,\label{espandoomeghino}
\end{align}
where $e_{i}\in \Z^{\Z}$ are the vectors of the canonical basis and
$\s_l= {\rm sign}(\al_{m_l}-\bt_{m_l})$.

\begin{remark}{\bf (Asymptotic).}\label{rmk:asy}
Recalling \eqref{dispLawWave}
 we note that, for $|j|\neq0$
 \begin{equation}\label{eq:asy}
 \omega_{j}=\sqrt{|j|^2+\tm}=|j|\sqrt{1+\frac{\tm}{|j|^{2}}}=|j|+\mathtt{r}_{j}(\tm)\,,\qquad |\tr_{j}(\tm)|\leq \frac{\tm}{2|j|}\leq 1\,.
 \end{equation}
\end{remark}
Theorem \ref{thm:mainVERO} is a direct consequence of the following.

\begin{proposition}\label{prop:stimeimproved}
Consider $\ell=\alpha-\beta$, $(\alpha,\beta)\in \Lambda$ (see \eqref{restrizioni indici}, \eqref{def:resonantset})
and let $m=m(\ell)$ according to Definition \ref{n star}.
Assume that the vector $\omega$ in \eqref{dispLawWave} belongs to the set 
$\mathtt{P}_{\gamma}$  in \eqref{diofSetPRIMA}.
Then 
one has the improved bound
\begin{equation}\label{goodsmalldiv}
|\omega\cdot\ell|\geq  \gamma^{{\tau^2(\ell)}}
\frac{1}{(\mathtt{C}N)^{6\tau(\ell)}}
\prod_{\substack{n\in\mathbb{Z} \\ n\neq m_1(\ell),m_2(\ell)}} 
\frac{1}{(1+|\ell_{n}|^{2}\langle n\rangle^{2})^{{4\tau^2(\ell)}}}\,,
\end{equation}
for some positive pure constant $\mathtt{C}$ large enough.
\end{proposition}

\begin{proof}
First of all, since $\omega\in \mathtt{P}_{\gamma}$, we note that estimate \eqref{diofSetPRIMA}, specialized for 
$p=0$, implies
\begin{equation}\label{stimaemmeelle}
|\omega\cdot\ell|\geq \gamma^{\td(\ell)}
\prod_{\substack{n\in\mathbb{Z} \\ n\neq m_1,m_2}} 
\frac{1}{(1+|\ell_{n}|^{2}\langle n\rangle^{2})^{\tau}}
\cdot \frac{1}{(1+|\ell_{m_1}|\langle m_{1}\rangle^{2})^{\tau(\ell)}}
\cdot \frac{1}{(1+|\ell_{m_2}|\langle m_{2}\rangle^{2})^{\tau(\ell)}}\,.
\end{equation}
Roughly speaking  our aim is to eliminate the dependence on the indexes $m_1,m_2$ in the r.h.s. of \eqref{stimaemmeelle}.

{\bf Step 1. ({No small divisors}).}  We claim that
\begin{equation}\label{divisorMAGG}
\sum_i (\al_i-\bt_i)|i| \geq 10 \sum_i |\al_i-\bt_i| \quad \Rightarrow\quad |\omega\cdot\ell|\geq 1\,,
\end{equation}
and the bound \eqref{goodsmalldiv} trivially follows.\\
Indeed, by Remark \ref{rmk:asy} and the triangular inequality, we deduce
\[
\begin{aligned}
\left|\sum_{i\in \mathbb{Z}}(\al_i-\bt_i) \omega_{j}\right|
&\geq 
\left|\sum_{i\in \mathbb{Z}}(\al_i-\bt_i)|i|\right|-
\left|\sum_{i\in \mathbb{Z}}(\al_i-\bt_i) \mathtt{r}_{i}(\tm)\right|
\\&\geq
10 \sum_i |\al_i-\bt_i|- \sum_i |\al_i-\bt_i|\geq1\,.
\end{aligned}
\]
As a consequence, in the following we shall always assume that
\begin{equation}\label{divisor}
\sum_i (\al_i-\bt_i)|i| \leq 10 \sum_i |\al_i-\bt_i|=10|\ell|\leq 10(D+|\alpha_0-\beta_0|)\stackrel{\eqref{cappella}}{\leq}
10N={10(|\alpha|+|\beta|)} \,.
\end{equation}

{\bf Step 2.}
First note that
\begin{equation}\label{globo1}
|\omega\cdot\ell |\leq \sum_{i\in\Z}|\alpha_{i}-\beta_{i}|\omega_{i}
\stackrel{\eqref{eq:asy}}{\leq}
\sum_{i\in\Z}|\alpha_{i}-\beta_{i}|(|i|+1)\leq 2\sum_{i\in\Z}|\alpha_{i}-\beta_{i}||i|
\stackrel{\eqref{divisor}}{\leq} 20 N\,.
\end{equation}
On the other hand, using $\omega_i \le |i| + 1 \le |m_3| + 1$, one has
\begin{equation}\label{globo2}
|\omega\cdot\widetilde{\ell}|\leq \sum_{\substack{i\in\Z \\ i\neq m_1,m_2}}|\alpha_{i}-\beta_{i}|\omega_{i}\leq {2} N\langle m_3\rangle\,.
\end{equation}
By  \eqref{espandoomeghino} we trivially have 
\begin{equation}\label{globo3}
\big| |\s_{1}\omega_{m_1}+\s_2\omega_{m_2}|-|\omega\cdot\widetilde{\ell}|\big|
\leq |\omega\cdot\ell|\stackrel{\eqref{globo1}}{\leq} 20 N\,.
\end{equation}
Then we have two cases: either $|\s_{1}\omega_{m_1}+\s_2\omega_{m_2}|
\leq |\omega\cdot\widetilde{\ell}|$,
or by \eqref{globo3} $|\s_{1}\omega_{m_1}+\s_2\omega_{m_2}|\leq |\omega\cdot\ell|+|\omega\cdot\widetilde\ell|$.

In both cases, by \eqref{globo1}, \eqref{globo2}, we deduce
\begin{equation}\label{globo4}
|\s_{1}\omega_{m_1}+\s_2\omega_{m_2}|\leq 20 N+{2}N\langle m_3\rangle
\leq {22} N\langle m_3\rangle\,,\qquad \forall \s_{1},\s_{2}=\pm1\,.
\end{equation}

\smallskip
\noindent
$\bullet$ In the case  $\s_{1}\s_2=1$, by \eqref{globo4} we get $|m_2|,|m_1|\leq 22N\langle m_3\rangle$.
If $\ell_{m_3}=0$ one would have  $\widetilde{\ell}\equiv0$. Therefore $|\omega\cdot\ell|=|\s_1\omega_{m_1}+\s_2\omega_{m_2}|$, which, 
in the case $\s_1\s_2=1$ is bounded from below by $1$ and \eqref{goodsmalldiv} follows trivially.
Assume now $|\ell_{m_3}|\geq1$.
Therefore
\[
\begin{aligned}
\frac{1}{(1+|\ell_{m_2}|\langle m_{2}\rangle^{2})^{\tau(\ell)}}&
\frac{1}{(1+|\ell_{m_1}|\langle m_{1}\rangle^{2})^{\tau(\ell)}}
\geq
\frac{1}{(1+N\langle m_{2}\rangle^{2})^{\tau(\ell)}}
\frac{1}{(1+N\langle m_{1}\rangle^{2})^{\tau(\ell)}} 
\\&
\geq \frac{1}{(1+{22}N^2\langle m_{3}\rangle^{2})^{\tau(\ell)}}
\frac{1}{(1+{22}N^2\langle m_{3}\rangle^{2})^{\tau(\ell)}} 
\\&\geq\frac{1}{{22N}^{4\tau(\ell)}}\frac{1}{(1+\langle m_{3}\rangle^{2})^{2\tau(\ell)}}
\geq\frac{1}{{22N}^{4\tau(\ell)}}\frac{1}{(1+|\ell_{m_3}|\langle m_{3}\rangle^{2})^{2\tau(\ell)}} \,,
\end{aligned}
\]
since $|\ell_{m_3}|\geq1$.
Then by \eqref{stimaemmeelle} we get
\[
|\omega\cdot\ell|\geq \gamma^{\td(\ell)}\frac{1}{N^{4\tau(\ell)}}
\prod_{\substack{n\in\mathbb{Z} \\ n\neq m_1,m_2}} 
\frac{1}{(1+|\ell_{n}|^{2}\langle n\rangle^{2})^{3\tau(\ell)}}
\]
which implies the bound \eqref{goodsmalldiv}.

\smallskip
\noindent
$\bullet$ In the case $\s_{1}\s_2=-1$ form \eqref{globo4}
we get no useful information on the size of $|m_1|,|m_2|$. We only get, recalling \eqref{eq:asy},
\begin{equation}\label{globo5}
|m_1|-|m_2|\leq 21N\langle m_3\rangle+2\leq {23}N\langle m_3\rangle\,.
\end{equation}
To deal with this case we reason differently.

\smallskip

 Without loss of generality we shall assume that $\s_1=1$, $\s_2=-1$.
By \eqref{espandoomeghino} and \eqref{eq:asy} we write
\[
\omega\cdot\ell=\omega\cdot\widetilde{\ell}+|m_1|-|m_2|+\mathtt{r}_{m_1}(\tm)-\mathtt{r}_{m_2}(\tm)\,.
\]
Since $\omega\in \mathtt{P}_{\gamma}$, by \eqref{diofSetPRIMA} applied with $p=|m_1|-|m_2|$
we get
\begin{equation}\label{fineprime}
\big|\omega\cdot\widetilde{\ell}+|m_1|-|m_2| \big|\geq 
\gamma^{\mathtt{d}(\widetilde\ell)}\prod_{n\in\mathbb{Z}} 
\frac{1}{(1+|\widetilde{\ell}_{n}|^{2}\langle n\rangle^{2})^{\tau(\widetilde{\ell})}}
\geq \gamma^{\mathtt{d}(\ell)}\prod_{\substack{n\in\mathbb{Z} \\ n\neq m_1,m_2 }} 
\frac{1}{(1+|\ell_{n}|^{2}\langle n\rangle^{2})^{\tau}}\,,
\end{equation}
where in the last inequality we used the definition of $\widetilde{\ell}$.
Moreover,
by definition we have
\[
\begin{aligned}
\mathtt{r}_{m_1}(\tm)-\mathtt{r}_{m_2}(\tm)&=\sqrt{|m_1|^{2}+\tm}-|m_1|-\sqrt{|m_2|^2+\tm}+|m_2|
\\&=
\frac{\tm}{\sqrt{|m_1|^{2}+\tm}+|m_1|}-\frac{\tm}{\sqrt{|m_2|^{2}+\tm}+|m_2|}
\\&
=\tm
\frac{\sqrt{|m_1|^{2}+\tm}+|m_1|-\sqrt{|m_2|^{2}+\tm}-|m_2|}{(\sqrt{|m_1|^{2}+\tm}+|m_1|)(\sqrt{|m_2|^{2}+\tm}+|m_2|)}\,.
\end{aligned}
\]
We note that 
\[
|\sqrt{|m_1|^{2}+\tm}-\sqrt{|m_2|^{2}+\tm}|\leq \frac{(|m_1|+|m_2|)(|m_1|-|m_2|)}{\sqrt{|m_1|^{2}+\tm}+\sqrt{|m_2|^{2}+\tm}}\leq |m_1|-|m_2|\,.
\]
The latter two bounds, together with \eqref{globo5}, imply
\[
|\mathtt{r}_{m_1}(\tm)-\mathtt{r}_{m_2}(\tm)|\leq \frac{92N\langle m_3\rangle}{\langle m_1\rangle\langle m_2\rangle}
\]
Let us now assume that
\begin{equation}\label{tagliom1}
|m_1|\geq   {184N\langle m_{3}\rangle}
\left(\gamma^{\mathtt{d}(\ell)}\prod_{\substack{n\in\mathbb{Z} \\ n\neq m_1,m_2 }} 
\frac{1}{(1+|\ell_{n}|^{2}\langle n\rangle^{2})^{\tau}}\right)^{-1/2}\,.
\end{equation}
By \eqref{globo5}
we also deduce that
\[
|m_2|\geq   \left(\gamma^{\mathtt{d}(\ell)}\prod_{\substack{n\in\mathbb{Z} \\ n\neq m_1,m_2 }} 
\frac{1}{(1+|\ell_{n}|^{2}\langle n\rangle^{2})^{\tau}}\right)^{-1/2}\,.
\]
Therefore
\begin{equation*}
|\mathtt{r}_{m_1}(\tm)-\mathtt{r}_{m_2}(\tm)|\leq \frac{1}{2}\gamma^{\mathtt{d}(\ell)}\prod_{\substack{n\in\mathbb{Z} \\ n\neq m_1,m_2 }} 
\frac{1}{(1+|\ell_{n}|^{2}\langle n\rangle^{2})^{\tau}}\,.
\end{equation*}
which combined with \eqref{fineprime} implies
\[
|\omega\cdot\ell|\geq  \gamma^{\td(\ell)}
\frac{1}{2}
\prod_{\substack{n\in\mathbb{Z} \\ n\neq m_1(\ell),m_2(\ell)}} 
\frac{1}{(1+|\ell_{n}|^{2}\langle n\rangle^{2})^{\tau(\ell)}}\,.
\]
Hence \eqref{goodsmalldiv} follows also in this case.

{\bf Step 4.} It remains to study the case
\begin{equation}\label{tagliom1opposto}
|m_1|, |m_1|\leq 184N\langle m_{3} \rangle\left(\gamma^{\mathtt{d}(\ell)}\prod_{\substack{n\in\mathbb{Z} \\ n\neq m_1,m_2 }} 
\frac{1}{(1+|\ell_{n}|^{2}\langle n\rangle^{2})^{\tau}}\right)^{-1/2}\,.
\end{equation}
Then it is easy to check that
\[
\begin{aligned}
\frac{1}{(1+|\ell_{m_2}|\langle m_{2}\rangle^{2})^{\tau(\ell)}}&
 \frac{1}{(1+|\ell_{m_1}|\langle m_{1}\rangle^{2})^{\tau(\ell)}} 
\geq \frac{1}{N^{2\tau(\ell)}}
 \frac{1}{(1+\langle m_{2}\rangle^{2})^{\tau(\ell)}}
 \frac{1}{(1+\langle m_{1}\rangle^{2})^{\tau(\ell)}} 
\\&\geq 
\frac{1}{(\mathtt{C} N)^{6\tau(\ell)}}
\gamma^{\tau(\ell)\mathtt{d}(\ell)}\prod_{\substack{n\in\mathbb{Z} \\ n\neq m_1,m_2 }} 
\frac{1}{(1+|\ell_{n}|^{2}\langle n\rangle^{2})^{{4\tau^2(\ell)}}}\,,
\end{aligned}
\]
for some pure constant $\mathtt{C}\geq184$,
and hence \eqref{stimaemmeelle} becomes
\[
|\omega\cdot\ell|\geq \gamma^{{\tau^2(\ell)}}\frac{1}{(\mathtt{C}N)^{6\tau(\ell)}}
\prod_{\substack{n\in\mathbb{Z} \\ n\neq m_1,m_2}} 
\frac{1}{(1+|\ell_{n}|^{2}\langle n\rangle^{2})^{{4\tau^2(\ell)}}}\,.
\]
This implies \eqref{goodsmalldiv} and the thesis.
\end{proof}

\section{Homological equation}

In this section we provide an application of the diophantine estimates given in section \ref{sec:smalldiv}.
In particular we provide precise estimates on the inverse of the Lie derivative $L_{\omega}$ (see \eqref{def:elleomega})
associate to the Hamiltonian
$\mathtt{D}_{\omega} := \sum_j \omega_j|u_j|^2$
where $\omega_{j}$ is in \eqref{dispLawWave}.

We first introduce, in subsection \ref{sec:classe} the class of Hamiltonians we deal with, and the we show how to solve an homological equation. This is the content of section \ref{sec:homo}

\subsection{Functional setting and  homogeneous Hamiltonians}\label{sec:classe}
In the following we  identify  $ L^2(\T,\C)$ 
with  the Banach space 
$\cF(\ell^2(\C))$ of $2\pi$-periodic functions 
\[
u(x)  = \sum_{j\in\Z} u_j e^{ijx}\,,\qquad u_{j}=\frac{1}{2\pi}\int_{\T}u(x)e^{-{\rm i}jx}{\rm d}x
\]
such that their Fourier's coefficients $(u_j)_{j\in\Z}\in \ell^2(\C)$. 
We shall always work with quite regular functions.
Let  $\mathtt w=(\mathtt w_j)_{j\in\Z}$ be  the  real sequence  
\begin{align}
&\tw=\tw(s,p):= \pa{\jap{j}^{ p}e^{ s\jap{j}^{\theta}}}_{j\in\Z}\,,\quad \theta\in(0,1)\,,
\label{peso sub}
\end{align}
and  let us set the 
Hilbert space
\begin{equation}\label{pistacchio}
\th_{\mathtt w}:=
\Big\{u:= (u_j)_{j\in\Z}\in\ell^2(\C)\; : \; \abs{u}_{\tw}^2:= 
\sum_{j\in\Z} \mathtt w_j^2 \abs{u_j}^2 < \infty
\Big\}\,,
\end{equation} 
endowed with the scalar product
\begin{equation*}
(u,v)_{\th_\tw}:=\sum_{j\in\Z} \mathtt w_j^2 u_j \bar v_j\,,\qquad u,v\in\th_{\tw}\,.
\end{equation*}
Moreover, given $r>0$, we denote by $B_r(\th_{\mathtt w})$ 
the closed
ball of radius $r$ centred at the origin of $\th_{\mathtt w}$.

Since the $\th_{\mathtt{w}}$ are 
contained in $\ell^{2}(\C)$, we endow them with the standard symplectic 
structure coming from the Hermitian product on $\ell^{2}(\C)$. Fix the symplectic structure to be\footnote{here the one form are defined the identification between $\C$
and $\R^{2}$, given by
\[
d u_j = \frac{1}{\sqrt 2}({d x_j+ \imm d y_j})\,,
\quad 
d \bar u_j = \frac{1}{\sqrt 2}({d x_j- \imm d y_j})\,,
\qquad u_{j}=\frac{1}{\sqrt{2}}(x_{j}+\ii y_{j})
\]}
\[
\Omega=\imm \sum_{j\in\Z} du_j\wedge d\bar{u}_j\,.
\]
Given a regular Hamiltonian $H: B_{r}(\mathtt{h}_{\tw})\to \R$ its Hamiltonian vector field is given by\footnote{we use the notation
\[
\frac{\partial}{\partial  u_j} =  
\frac{1}{\sqrt 2}\big({\frac{\partial}{\partial x_j} 
- \imm \frac{\partial}{\partial y_j}}\big)\,,
\quad  
\frac{\partial}{\partial  \bar u_j} =  
\frac{1}{\sqrt 2}\big({\frac{\partial}{\partial x_j} 
+ \imm \frac{\partial}{\partial y_j}}\big)\,.
\]}
\[
X_{H}^{(j)}=-\ii \frac{\partial}{\partial\bar{u}_{j}}H\,.
\]
We introduce the following class of Hamiltonians.

\begin{definition}\label{def:admiHam}
Let $r>0$ and consider a Hamiltonian
$ H : B_r(\th_{\mathtt w}) \to \R$
such that there exists a pointwise  
absolutely convergent power series expansion\footnote{As usual given 
a vector $k\in \Z^\Z$, $|k|:=\sum_{j\in\Z}|k_j|$.}
\begin{equation}\label{HamPower}
H(u)  = 
\sum_{\substack{
\al,\bt\in\N^\Z\,, \\2\leq |\al|+|\bt|<\infty} }
\!\!\!
H_{\al,\bt}u^\al \bar u^\bt\,,
\qquad
u^\al:=\prod_{j\in\Z}u_j^{\al_j}\,.
\end{equation}

\noindent
$\bullet$ {\bf (Admissible Hamiltonians).}
We say that $H$ as in \eqref{HamPower}
is \emph{admissible} if  the following properties hold: 
\begin{enumerate}
\item \emph{Reality condition}:
\begin{equation}\label{real}
H_{\al,\bt}= \overline{ H}_{\bt,\al}\,,\qquad \forall\, \al,\bt\in\mathbb{N}^{\Z}\,;
\end{equation}
\item \emph{Momentum conservation}:	
\begin{equation}\label{momento}
H_{\al,\bt}\neq0\quad \Rightarrow\quad
\pi(\al - \bt) := \sum_{j\in\Z}j\pa{\al_j - \bt_j}= 0\,.
\end{equation}
\end{enumerate}	

\smallskip
\noindent
$\bullet$ {\bf (Regular Hamiltonians).} 
We say that $H$ as in \eqref{HamPower}
is \emph{regular} if  the following properties hold: 
\begin{enumerate}
\item $H$ is admissible;

\item the the \co{majorant} Hamiltonian
\begin{equation}\label{etamag}
\und { H} (u):= 
\sum_{(\alpha,\beta)\in\cM} \abs{{H}_{\alpha,\beta}}u^{\alpha}\bar{u}^{\beta}
\end{equation} 
is point-wise  absolutely convergent on $B_r(\th_{\mathtt w})$,
where we set
\begin{equation*}
\mathcal{M}:=
\left\{
(\alpha,\beta)\in \mathbb{N}^{\Z}\times\N^{\Z} : \pi(\al - \bt)= 0, 
|\alpha|+|\beta|<\infty
\right\}\,;
\end{equation*}

\item one has
\begin{equation}\label{normaHamilto}
|H|_{r,\tw}
:=
r^{-1} \Big(
\sup_{\norm{u}_{{\mathtt w}}\leq r} 
\norm{{X}_{{\underline H}}}_{{\mathtt w}} 
\Big) < \infty\,.
\end{equation}
\end{enumerate}	
 We denote by $\cH_{r}(\th_{\mathtt w})$ the space of regular Hamiltonians.
 
 \smallskip
 $\bullet$ {\bf (Scaling degree).} Given $\td \in\N$, let  $\cH^{(\td)}$  be the vector space 
of homogeneous polynomials of degree $\td + 2$, 
that is admissible  Hamiltonians of the form
\[
\sum_{\substack{(\al,\bt)\in\mathcal{M} \\ |\al| + |\bt| = \td + 2}} H_{\al,\bt} u^{\al} \bar{u}^{\bt}\,.
\]
We shall say that a regular Hamiltonian $H$ has 
\emph{scaling degree} $\td=\td(H)$ if  $H\in \cH_{r}(\th_{\mathtt w})\cap\cH^{(\td)}$.
\end{definition}

\begin{remark}
\label{embeddignsspaces}
We remark the following facts:

\noindent
$\bullet$
Given two positive sequences $\tw = \pa{\tw_j}_{j\in\Z},\tw' = (\tw'_j)_{j\in\Z}$
we write that $\tw\leq \tw'$ if the inequality holds
point wise, namely
\[
\tw\leq \tw' \quad
\iff\quad
\tw_j\leq \tw'_j\,,\ \ \ \forall\, j\in\Z\,.
\]
In this way if $r'\le r$ and $\tw\leq \tw'$ 
then $B_{r'}(\th_{\mathtt w'}) \subseteq B_r(\th_\tw)$.

\noindent
$\bullet$ If a Hamiltonian $H$ satisfies \eqref{real},  it means that it 
is real analytic in the real and imaginary part of $u$.

\noindent
$\bullet$ If a Hamiltonian $H$ satisfies \eqref{momento} then it 
Poisson commutes with $\sum_{j\in\Z} j\abs{u_j}^2$ where, 
given two admissible Hamiltonians $H,G$,
the Poisson brackets are given by 
\begin{equation}\label{poipoisson}
\{H,G\}={\rm i}
\sum_{j\in\Z}
\big(\partial_{u_j}G\partial_{\bar{u}_j}H-
\partial_{\bar{u}_j}G\partial_{u_j}H\big)\,.
\end{equation}

\noindent
$\bullet$ The Hamiltonian functions being defined modulo a constant term, 
we shall assume without loss of generality that $H(0)=0$. 
\end{remark}

\noindent
An important remark is that
that all the dependence on the parameters $r,\tw$ 
of the  norm in \eqref{normaHamilto}  can be {\it encoded}
in the coefficients
\begin{equation}\label{persico}
c^{(j)}_{r,\tw}(\al,\bt)
:=
r^{|\al|+|\bt|-2}
\frac{\tw_j^2}{\tw^{\al+\bt}}\,,\qquad \tw^{\al+\bt} 
= \prod_{j\in\Z}\tw_j^{\al_j+\bt_j}\,,
\end{equation}
defined for any $\al,\bt\in\N^\Z$ and $j\in\Z$.
In view of our choices of the weights in \eqref{peso sub} 
we have that the coefficients in \eqref{persico} have the following form:
\begin{align}
&	
c^{(j)}_{r,\tw}(\al,\bt)=r^{|\al|+|\bt|-2}\frac{\jap{j}^{ 2p}}{\prod_{i\in\Z}\jap{j}^{p(\al_i+\bt_i)}}
e^{ s\big(2\jap{j}^{\theta}-\sum_{i\in\Z}(\al_i+\bt_i)\jap{i}^{\theta} \big)}\,.\label{coeffSE}
\end{align}

To formalize such \emph{encoding} we reason as follows.
For any $H\in \cH_{r}(\th_{\mathtt w})$ we define  a map
\[
B_1(\ell^2)\to \ell^2 \,,\quad y=\pa{y_j}_{j\in \Z }\mapsto 
\pa{Y^{(j)}_{H}(y;r,\tw)}_{j\in \Z}
\]
by setting
\begin{equation}\label{giggina}
Y^{(j)}_{H}(y;r,\tw) := 
\sum_{(\al,\bt)\in\mathcal{M}} |H_{\al,\bt}| \frac{(\al_j+\bt_j)}{2}c^{(j)}_{r,\tw}(\al,\bt) y^{\al+\bt-e_j}
\end{equation}
where  $e_j$ is the $j$-th basis vector in $\N^\Z$, while the coefficient
$c^{(j)}_{r,\tw}(\al,\bt)$ is defined right above in \eqref{persico}.
The following properties give a systematic way for 
computing the norm of a given Hamiltonian 
and its relation w.r.t. another one.

\begin{lemma} \label{norme proprieta} 
Let $\ri, r'>0,$ $\twi,\twf\in \R_+^\Z$. 
The following properties hold.

\vspace{0.3em}
\noindent
$(1)$ The norm of $H$ can be expressed as
\begin{equation}\label{ypsilon}
\abs{H}_{r,\tw}= 
\sup_{|y|_{\ell^2}\le 1}\abs{Y_H(y;r,\tw)}_{\ell^2}\,.
\end{equation}

\vspace{0.3em}
\noindent
$(2)$ Given $H^{(1)}\in \cH_{r',\twf}$ and $H^{(2)}\in \cH_{\ri,\twi}$\,,
such that $\forall\, \al,\bt\in \N^\Z$ and  $j\in \Z$ with $\al_j+\bt_j\neq 0$ 
one has
\begin{equation}\label{alberellobello}
|H^{(1)}_{\al,\bt}| c^{(j)}_{\rf,\twf}(\al,\bt)  \le 
c|H^{(2)}_{\al,\bt}| c^{(j)}_{\ri,\twi}(\al,\bt),
\end{equation}
for some $c>0$, then
\[
|H^{(1)}|_{\rf,\twf}
\le c |H^{(2)}|_{\ri,\twi}\,.
\]

\vspace{0.3em}
\noindent
$(3)$ {\bf (Immersion).} For any $p>1, s>0$ the norm $\norm{\cdot}_{r,\tw}$ 
with $\tw=\tw(s,p)$ (see \eqref{peso sub})
is monotone increasing in $r$. Moreover, letting $r>0$,
one has, 
for any $\s,s>0$, that
\begin{equation}\label{emiliaparanoica}
|H|_{r,\tw(s+\s,p)} \le  |H|_{r,\tw(s,p)}\,.
\end{equation}
\end{lemma}
\begin{proof}
It follows by Lemma $3.1$ in \cite{BMP:CMP}. 
\end{proof}

\subsection{Inversion of the adjoint action}\label{sec:homo}

Given a diophantine vector $\omega\in \mathtt{P}_{\gamma}$, 
in view of Remark \ref{rmk:ker2} and 
by formula \eqref{def:adjaction} we  deduce that
\[
L_{\omega}H=0 \qquad \Leftrightarrow\qquad  H\in \cK_r(\th_\tw):=\Big\{H\in\cH_{r}(h_{\tw})\, : 
\sum_{(\al,\bt) \in\mathtt{R}} H_{\al,\bt} u^{\al}\bar{u}^{\bt}\Big\}\,.
\]

\noindent
Hence the operator $L_{\omega}$ is formally invertible
when acting on the subspace  
\begin{equation}\label{range2}
\cR_r(\th_\tw)=\cK_r(\th_\tw)^{\perp}:=
\Big\{
H\in\cH_{r}(h_{\tw})\, : 
\sum_{(\al,\bt) \in\mathtt{R}^{c}} H_{\al,\bt} u^{\al}\bar{u}^{\bt}
\Big\}\,,
 \end{equation}
containing  those Hamiltonians supported on monomials
$u^{\alpha}\bar{u}^{\beta}$ with $(\alpha,\beta)\in \mathtt{R}^{c}$.
We decompose the space of regular Hamiltonians $\cH_r(\th_\tw)$
as
\[
\cH_r(\th_\tw) = \cK_r(\th_\tw)\oplus \cR_r(\th_\tw)\,,
\]
and we denote by $\Pi_{\cK}$ and $\Pi_{\cR}$
 the  continuous projections	 
 on the subspaces $\cK_r(\th_\tw)$, $ \cR_r(\th_\tw)$.

\noindent 
Obviously, for {diophantine}  frequency,
$\cR^\wc_{r}(\th_\tw)$ and	 $\cK^\wc_{r}(\th_\tw)$
represent the range and kernel 
of  $L_\omega$ respectively.

The key result of this section is the following.

\begin{theorem}{\bf (Inverse of the adjoint action).}\label{shulalemma}
Fix $\mathtt{N}\in \mathbb{N}$, $r>0$, $p>1$ and $s>0$.
Consider $\mathtt{w}(s,p)$ in \eqref{peso sub}
and a Hamiltonian function $f\in \mathcal{R}_r(\th_\tw)\cap \mathcal{H}^{(\tN)}$ (see Def. \ref{def:admiHam} and recall \eqref{range2}). 
For any $\omega\in \mathtt{P}_{\gamma}$ 
(see \eqref{diofSetPRIMA}) 
the following holds.

\smallskip
\noindent
There exists an absolute constant $\mathtt{C}>0$ (independent of $\tN$) such that for any 
 $0<\s\ll1$ one has
that 
\[
|L_{\omega}^{-1}f|_{r,\mathtt{w}(p,s+\s)}\leq J_0(\s,\tN) |f|_{r,\mathtt{w}(p,s)}\,,
\]
where  $L_{\omega}$ is in \eqref{def:adjaction} and 
\begin{equation}\label{controlJ0}
J_0(\s,\tN)
:=\gamma^{- \mathtt{c} \tN^2}\exp\Big(\Big(\frac{\mathtt{c}\tN^{4}}{\s\theta(1-\theta)}\Big)^{2+\frac{1}{\theta}} \Big)
\,.
\end{equation}
\end{theorem}

In order to prove the Theorem above we first recall this simple result.

\begin{lemma}\label{lem:constance2SE}
For all $(\al,\bt)\in \mathcal{M}$ (see \eqref{mass-momindici}) the following holds.

\noindent
$(i)$ If 
\begin{equation}\label{divisorBistris}
\sum_i (\al_i-\bt_i)|i| \le 10 \sum_i |\al_i-\bt_i| \,,
\end{equation}
then, setting $m=m(\alpha-\beta)$, we have for any $j$ such that $\alpha_j+\beta_{j}\neq0$
\begin{equation}\label{constance2SE}
{\sum_{i\neq m_1,m_2}\langle i\rangle^{\theta}|\alpha_i-\beta_i|\leq} \sum_{l\ge 3} \na_l^{\theta} \le 
\frac{1}{1-\theta}\pa{\sum_i \pa{\al_i+\bt_i}\jap{i}^{\theta}- 2\jap{j}^{\theta}}\,.
\end{equation}

\noindent
$(ii)$ If on the contrary \eqref{divisorBistris} does not hold then
\begin{equation}\label{pool1}
|\omega\cdot(\alpha-\beta)|\geq1\,,
\end{equation}
where $\omega$ is given in \eqref{dispLawWave}.
\end{lemma}
\begin{proof}
The first inequality is a direct consequence of formula \eqref{pula2}. Indeed,
recalling Def. \ref{n star},
we have
\[
\sum_{i\neq m_1,m_2}\langle i\rangle^{\theta}|\alpha_i-\beta_i|\leq 
|\alpha_0-\beta_0|+\sum_{j\geq3}\langle m_{j}\rangle^{\theta}
\stackrel{\eqref{abbacchio}}{\leq}\sum_{j\geq3}\langle n_{j}\rangle^{\theta}\,.
\]
The second inequality 
is proved in \cite{Bourgain:1996, Bourgain:2005}, and \cite{CLSY}.

\noindent
Item $(ii)$ can be deduced form Step $1$ in the proof of Proposition \ref{prop:stimeimproved}.
\end{proof}

\begin{proof}[{\bf Proof of Theorem \ref{shulalemma}}]

\noindent
Since, by hypothesis, $f$ 
belongs  to the range of the operator $L_{\omega}$, the Hamiltonian
$L_{\omega}^{-1}f$ is well-defined with coefficients given by 
\[
(L_{\omega}^{-1}f)_{\al,\bt}= \frac{f_{\al,\bt}}{-\imm \omega\cdot(\al-\beta)}\,,\qquad 
\forall \,(\al,\bt)\in\mathtt{R}^{c}\,. 
\]
Recall the coefficients in 
\eqref{coeffSE}. In view of property \eqref{alberellobello} with 
$\tw' = \tw(p,s+\s)$ and $\tw = \tw(p,s)$ and formula \eqref{def:adjaction} ,
we have that, if 
\begin{equation*}
 J_0 := 
 \sup_{\substack{j\in\Z,\,  (\al,\bt)\in\Lambda \\ 
 \al_j+\bt_j\neq 0 \\ |\al-\bt|\leq \tN+2}}
\frac{c^{(j)}_{\ri,\mathtt{w}(s+\s,p)}(\al,\bt) }
{c^{(j)}_{\ri,\mathtt{w}(s,p)}(\al,\bt) |\omega\cdot (\al-\bt)|}<+\infty\,,
\end{equation*}
then
\[
|L_{\omega}^{-1}f|_{r,\mathtt{w}(p,s+\s)}\leq J_0 |f|_{r,\mathtt{w}(p,s)}\,.
\]
Therefore in order to get the result it is sufficient to estimate  the quantity $J_0$.

By an explicit computation using \eqref{coeffSE} we get
\[
J_0 =  
\sup_{\substack{j\in\Z,\,  (\al,\bt)\in\Lambda \\ 
\al_j+\bt_j\neq 0  \\
 |\al-\bt|\leq \tN+2}}
\frac{e^{-\s\pa{\sum_i\jap{i}^{\theta} (\al_i+\bt_i) -2\jap{j}^{\theta}}}}{\abs{\omega\cdot{\pa{\al - \bt}}}}\,.
\]
By Lemma \ref{lem:constance2SE}-(ii), 
we just have to study the case in which \eqref{divisorBistris} holds true.
Recall also that, by Proposition \ref{prop:stimeimproved}, for $\omega\in\mathtt{P}_{\gamma}$, one has the estimates 
\eqref{goodsmalldiv} on the small divisor $|\omega\cdot(\alpha-\beta)|$.

In the following we do not keep track of absolute constants and just denote them by 
$\mathtt{c}$, which may change along the proof.
Since $\ell=\al-\bt$, $|\ell|\leq \tN+2$ we notice that 
$\mathtt{d}=\mathtt{d}(\ell)\leq 4\tN$ and 
$\tau=\tau(\ell)\leq 36 \tN^2$.
Therefore we have
\[
\begin{aligned}
J_0 &\leq 
\frac{(\mathtt{C}(\tN+2))^{6\tau(\ell)}}{\gamma^{{\tau^2(\ell)}}} \exp\left(-
\s\Big(\sum_{i\in\Z} \jap{i}^{\theta} (\al_i+\bt_i) -2\jap{j}^{\theta} \Big)
+{\sum_{\substack{i\in\Z \\ i\neq m_1,m_2}} 4\tau^2(\ell)\ln(1+(\al_i-\bt_i)\langle i\rangle^{2} ) }
\right)
\\
&\stackrel{\mathclap{\eqref{constance2SE}}}{\leq}
\gamma^{- \mathtt{c} \tN^4}\tN^{\mathtt{c} \tN^2}
\exp \left(\sum_{\substack{i\in\Z \\ i\neq m_1,m_2}}\left[-{\s (1-\theta) } \abs{\al_i - \bt_i}\jap{i}^{\theta})\right]
+ \mathtt{c}\tN^4 \ln{\pa{1 + \pa{\al_i - \bt_i}^{2}\jap{i}^{{2}}}} \right) 
\\
&\le \gamma^{- \mathtt{c} \tN^2}\tN^{\mathtt{c} \tN^2}  \exp{ (-\mathtt{c}\tN^4 \sum_{\substack{i\in\Z \\ i\neq m_1,m_2}}H_i(|\al_i - \bt_i|)})\,,
\end{aligned}
\]
where for $0<\s\leq 1$, $i\in \Z$ , we defined 
\begin{equation}\label{defalpha}
H_i(x) := \alpha x\jap{i}^{\theta} - \ln{\pa{1 + x^2\jap{i}^2}}\,,\qquad \alpha:=\frac{\s (1-\theta) }{ \mathtt{c}\tN^4} \,,
\end{equation}
where $x:= |\al_i - \bt_i| \geq 1$.
We observe that there exists $X(\al)$ such that the following inequalities hold: 
\begin{equation}\label{albero1bis}
\al x \jap{i}^\theta - \ln{\pa{1 + {x^2\jap{i}^2}}} 
\ge 0\,,
\quad \mbox{if}\quad \jap i \ge X(\al):=\left(\frac{2}{\alpha\theta}\right)^{\frac{1}{\theta}}\,. 
\end{equation}
This can be easily checked by an explicit computation by studying the 
function $f(y):=\alpha xy^{\theta}-\ln(1+x^2y^2)$ for $x,y\geq1$.

 Consequently
\begin{equation}\label{albero1}
J_0 \le \gamma^{- \mathtt{c} \tN^2}\tN^{\mathtt{c} \tN^2} 
\exp\big(- \mathtt{c}\tN^4 \inf_{ x\ge 1} \sum_{i:\jap i\le X(\al)} H_i(x)\big)\,.
\end{equation}
Let us compute $\inf_{x\ge 1}  H_i(x)$. We have
\begin{equation}\label{caldo1}
H_i(x) \ge \hat H_i(x)
- \ln{\pa{1 + \jap{i}^2}} \,,\qquad \hat H_i(x):= \al x \jap{i}^{\theta} - \ln{\pa{1 +x^2}}\,.
\end{equation}
We first  note that
\[
\hat H'_i(x) = \frac{\alpha \jap{i}^\theta(1+x^2)-2x}{1+x^2}\,,
\]
we then distinguish two cases.

\noindent
$(i)$ If one has that $\alpha\jap{i}^{\theta}>1$, then there are no solutions of $\hat H'_i(x)=0$. Therefore
we bound (see \eqref{caldo1})
\begin{equation}\label{caldo4}
\begin{aligned}
\inf_{x\ge 1}  H_i(x)&\geq \alpha\jap{i}^{\theta}-\ln(2)-\ln(1+\jap{i}^{2})\stackrel{\jap{i}\leq X(\alpha)}{\geq}-\ln(2)-\ln(1+X^2(\alpha))
\\&\stackrel{\eqref{albero1bis}}{\geq}
-\ln2-\ln\Big(1+\big(\frac{2}{\alpha\theta}\big)^{\frac{2}{\theta}}\Big)\,.
\end{aligned}
\end{equation}

\smallskip
\noindent
$(ii)$ One the contrary assume $\alpha\jap{i}^{\theta}\leq1$. This means that one must have
\begin{equation}\label{caldo2}
1\leq \jap{i}\leq \left(\frac{1}{\alpha}\right)^{\frac{1}{\theta}}\stackrel{\eqref{albero1bis}}{\leq}
\left(\frac{2}{\alpha\theta}\right)^{\frac{1}{\theta}}\,.
\end{equation}
In this case we have $\hat H'_i(x)=0$ for 
\[
x=\frac{1+\sqrt{1-(\alpha\jap{i}^{\theta})^2}}{\alpha\jap{i}^{\theta}}\,.
\]
Therefore we have the bound
\begin{equation}\label{caldo3}
\begin{aligned}
\inf_{x\geq1}H_i(x) &\stackrel{\eqref{caldo1}}{\ge}1+\sqrt{1-(\alpha\jap{i}^{\theta})^2} 
- \ln{\pa{1 +\Big(\frac{1+\sqrt{1-(\alpha\jap{i}^{\theta})^2}}{\alpha\jap{i}^{\theta}}\Big)^2}}-\ln(1+\jap{i}^{2})
\\&
\stackrel{\eqref{caldo2}}{\geq}1-\ln\big(1+\frac{4}{\alpha^2}\big)-\ln\Big(1+\big(\frac{2}{\alpha\theta}\big)^{\frac{2}{\theta}}\Big)
\geq -\ln\left[\Big(1+\frac{4}{\alpha^2}\Big)\Big(1+\frac{2}{\alpha\theta}\Big)  \right]^{\frac{2}{\theta}}
\\&\geq
-\ln\Big(1+\frac{16}{\alpha^3\theta}\Big)^{\frac{2}{\theta}}\,.
\end{aligned}
\end{equation}
By \eqref{caldo4}-\eqref{caldo3} we infer
 that
 \[
 \inf_{x\geq1}H_i(x) \geq -\ln\Big(1+\frac{16}{\alpha^3\theta}\Big)^{\frac{2}{\theta}}\,.
\]
The latter bound, together with \eqref{albero1}, implies 
\[
\begin{aligned}
J_0 &\leq\gamma^{- \mathtt{c} \tN^2}\tN^{\mathtt{c} \tN^2} \exp\Big(\mathtt{c}\tN^4 X(\alpha)
\ln\Big(1+\frac{16}{\alpha^3\theta}\Big)^{\frac{2}{\theta}}\Big)
\\&
=\gamma^{- \mathtt{c} \tN^2}\tN^{\mathtt{c} \tN^2}
\Big(1+\frac{16}{\alpha^3\theta}\Big)^{\frac{2}{\theta}\mathtt{c}\tN^4 X(\alpha)}
\\&\stackrel{\eqref{defalpha}}{=}
\gamma^{- \mathtt{c} \tN^2}\tN^{\mathtt{c} \tN^2}
\exp\Big( \frac{\mathtt{c}\tN^{4}}{\theta}\Big(\frac{\mathtt{c}\tN^{4}}{\s\theta(1-\theta)}\Big)^{\frac{1}{\theta}} 
\ln\Big(1+\Big(\frac{\mathtt{c}\tN^4}{\s\theta(1-\theta)}\Big)^{3}\Big) \Big)
\\&=\gamma^{- \mathtt{c} \tN^2}\tN^{\mathtt{c} \tN^2}
\exp\Big(\s(1-\theta) \frac{\mathtt{c}\tN^{4}}{\s\theta(1-\theta)}\Big(\frac{\mathtt{c}\tN^{4}}{\s\theta(1-\theta)}\Big)^{\frac{1}{\theta}} 
\ln\Big(1+\Big(\frac{\mathtt{c}\tN^4}{\s\theta(1-\theta)}\Big)^{3}\Big) \Big)
\\&\leq
\gamma^{- \mathtt{c} \tN^2}\exp\big(\mathtt{c} \tN^2\ln(\tN)\big)
\exp\Big(\s(1-\theta)\Big(\frac{\mathtt{c}\tN^{4}}{\s\theta(1-\theta)}\Big)^{2+\frac{1}{\theta}} \Big)
\leq\gamma^{- \mathtt{c} \tN^2}\exp\Big(\Big(\frac{\mathtt{c}\tN^{4}}{\s\theta(1-\theta)}\Big)^{2+\frac{1}{\theta}} \Big)\,,
\end{aligned}
\]
taking $\mathtt{c}$ larger enough in each inequality.
This is the desired bound \eqref{controlJ0}.
\end{proof}

\vspace{0.5cm}
\gr{Acknowledgements.} 
The authors have been  supported by the  research project 
PRIN 2020XBFL ``Hamiltonian and dispersive PDEs" of the 
Italian Ministry of Education and Research (MIUR).  
R. Feola has been supported by the  research project PRIN 2022HSSYPN ``Turbulent Effects vs 
Stability in Equations from Oceanography''
of the 
Italian Ministry of Education and Research (MIUR).  
J.E. Massetti has been supported by 
the  research project 
PRIN 2022FPZEES ``Stability in Hamiltonian dynamics and beyond"
of the 
Italian Ministry of Education and Research (MIUR).  
The authors acknowledge the support of  INdAM-GNAMPA.

\vspace{0.5cm}

\gr{Declarations}. Data sharing is not applicable to this article as no datasets were generated or analyzed during the current study.

\noindent
Conflicts of interest: The authors have no conflicts of interest to declare.


\end{document}